\documentclass{amsart}
\usepackage{amsrefs}

\def \rr {\mathbb{R}}
\def \rn {\mathbb{R}^n}
\def \sn {\mathbb{S}^n}
\def \sd {\mathbb{S}^d}

\def \crit {2^\star}
\def \ue {u_\eps}

\def \xe {x_\eps}
\def \eps {\varepsilon}

\DeclareMathOperator{\can}{can}
\DeclareMathOperator{\Eucl}{Eucl}

\newtheorem{thm}{Theorem}[section]
\newtheorem{prop}{Proposition}[section]
\newtheorem{coro}{Corollary}[section]
\newtheorem{claim}{Claim}[section]

\date{September 8th, 2014.}

\title[Peaks and degenerate functions]{Sign-changing solutions to elliptic second order equations: glueing a peak to a degenerate critical manifold}

\author{Fr\'ed\'eric Robert}
\address{Fr\'ed\'eric Robert, Institut \'Elie Cartan, UMR 7502, Universit\'e de Lorraine, BP 70239, F-54506 Vand{\oe}uvre-l\`es-Nancy, France; Pacific Institute for the Mathematical Sciences, UMI CNRS 3069, 4176-2207 Main Mall, Vancouver, BC, V6T 1Z4, Canada}
\email{frederic.robert@univ-lorraine.fr}

\author{J\'er\^ome V\'etois}
\address{J\'er\^ome V\'etois, Universit\'e de Nice Sophia Antipolis, Laboratoire J.-A. Dieudonn\'e, CNRS UMR 7351, Parc Valrose, F-06108 Nice Cedex 2, France}
\email{vetois@unice.fr}

\thanks{The research of the first author is partially supported by INSMI (CNRS)}

\begin{document}

\begin{abstract}
We construct blowing-up sign-changing solutions to some nonlinear critical equations by glueing a standard bubble to a degenerate function. We develop a new method based on analyticity to perform the glueing when the critical manifold of solutions is degenerate and no Bianchi--Egnell type condition holds.
\end{abstract}

\maketitle

\section{Introduction and statement of the results}\label{sec:intro}

Let $(M,g)$ be a smooth compact Riemannian manifold of dimension $n\geq 3$, and let $h\in C^{0,\theta}(M)$ ($\theta\in (0,1)$) be such that $\Delta_g+h$ is coercive where $\Delta_g=-\hbox{div}_g(\nabla)$ is the Laplace-Beltrami operator. In \cite{robertvetois}, we addressed the question of the existence of a family $(\ue)_{\eps>0}\in C^{2,\theta}(M)$ of blowing-up solutions of type $(u_0-B)$ to 
\begin{equation}\label{eq:ue}
\Delta_g \ue+h\ue=|\ue|^{\crit-2-\eps}\ue\hbox{ in }M,
\end{equation}
where $\crit:=\frac{2n}{n-2}$. Concerning terminology, we say that $(\ue)_{\eps}$ is of type $(u_0-B)$ when there exists a function $u_0\in C^{2,\theta}(M)$ positive that is a solution to 
\begin{equation}\label{eq:u0}
\Delta_g u_0+hu_0=u_0^{\crit-1}\hbox{ in }M
\end{equation} 
and such that
\begin{equation*}
\ue=u_0-B_\eps+o(1),
\end{equation*}
where $(B_\eps)_\eps$ is a bubble as defined in \eqref{def:bubble} below and $\lim_{\eps\to 0}o(1)=0$ in $H_1^2(M)$, the completion of $C^\infty(M)$ for the norm $u\mapsto \Vert u\Vert_{H_1^2}:=\Vert u\Vert_2+\Vert\nabla u\Vert_2$. Solutions of type $(u_0-B)$ are sign-changing. When $h\equiv c_n R_g$, where $c_n:=\frac{n-2}{4(n-1)}$ and $R_g$ is the scalar curvature, equation \eqref{eq:u0} is the Yamabe equation, and $\Delta_g+h$ is coercive if and only if $(M,g)$ has positive Yamabe invariant. There is an extensive literature on the existence of positive blowing-up solutions to equations of type \eqref{eq:ue}: see for instance Rey \cite{Rey} for a historical reference, Brendle--Marques \cite{BreMar} for the Yamabe equation, Druet--Hebey \cite{druethebeyTAMS} and Esposito--Pistoia--V\'etois \cite{EspPisVet} for perturbations of the Yamabe equation, Chen--Wei--Yan \cite{chenweiyan} and Hebey--Wei \cite{hw} for equations on the sphere, and the references therein. Sign-changing blowing-up solutions to \eqref{eq:ue} on the canonical sphere have been constructed by del Pino--Musso--Pacard--Pistoia \cites{delPMusPacPis1,delPMusPacPis2} and Pistoia--V\'etois \cite{PisVet}. We refer to Robert--V\'etois \cite{robertvetois} for a  discussion and references on the compactness of solutions to \eqref{eq:ue}.

\medskip\noindent In \cite{robertvetois}, we gave sufficient conditions to get blowing-up solutions of type $(u_0-B)$ to \eqref{eq:ue} provided that $u_0$ is a nondegenerate solution to \eqref{eq:u0}, that is $K_0=\{0\}$ where
\begin{equation}\label{def:K0}
K_0:=\{\varphi\in C^{2,\theta}(M)/\, \Delta_g \varphi+h\varphi=(\crit-1)u_0^{\crit-2}\varphi\hbox{ in }M\}.
\end{equation}
When $u_0$ is degenerate, the situation can be different. In \cite{robertvetois}, we showed that there is no blowing-up solutions of type $(u_0-B)$ to the constant scalar curvature equation on the canonical sphere: in this case, $u_0$ is necessarily degenerate. 

\medskip\noindent The present article is devoted to the analysis of the degenerate case, that is when $K_0\neq\{0\}$. We say that $u_0\in C^{2,\theta}(M)\setminus\{0\}$ is a strict local minimizer of $I_0$ if there exists $\nu>0$ such that
$$I_0(u)>I_0(u_0)\hbox{ for all }u\in B_\nu(u_0)\setminus \rr u_0\,,$$
where
$$I_0(u):=\frac{\int_M\left(|\nabla u|_g^2+hu^2\right)\, dv_g}{\left(\int_M|u|^{\crit}\, dv_g\right)^{\frac{2}{\crit}}}$$
for all $u\in H_1^2(M)\setminus \{0\}$. Our main result is the following:

\begin{thm}\label{th:1} 
We let $(M,g)$ be a compact Riemannian manifold of dimension $n\geq 3$ with positive Yamabe invariant and we fix $h\equiv \frac{n-2}{4(n-1)} R_g$. We assume that there exists $u_0\in C^{2,\theta}(M)$ that is a positive solution to \eqref{eq:u0} and a strict local minimizer of $I_0$. We assume either that $\{3\leq n\leq 9\}$ or that $\{(M,g)$ is locally conformally flat$\}$. Then there exists a solution of type $(u_0-B)$ to \eqref{eq:ue}.
\end{thm}
\noindent It follows from the compactness results of Schoen \cite{Sch3} and Khuri-Marques--Schoen \cite{KhuMarSch} (see also Druet \cite{druet:imrn}) that blowing-up solutions to \eqref{eq:ue} must change sign under the assumptions of Theorem \ref{th:1}.

\smallskip\noindent As a remark, any nondegenerate local minimizer of $I_0$ is a strict local minimizer, so we recover the main theorem of \cite{robertvetois}. Moreover no solution of the Yamabe equation on the sphere is a strict local minimizer. However, as soon as one takes the product of a sphere with another manifold, one gets examples of degenerate strict local minimizers. We refer to Section \ref{sec:ex} for such examples, in particular to Corollary \ref{th:type:4:yam}. %scalar curvature

\medskip\noindent We prove Theorem \ref{th:1} by performing a finite-dimensional reduction modeled on $(u-B)$ where $B$ is a bubble and $u\in \mathcal{M}$, and where $\mathcal M$ is a suitable finite-dimensional analytic manifold containing $u_0$. We construct the manifold $\mathcal M$ such that its elements are as close as possible to solutions of \eqref{eq:u0}: this is done using a first finite-dimensional reduction. The manifold $\mathcal M$ is locally parametrized by $K_0$, and the tangent space of $\mathcal M$ at $u_0$ is $K_0$. The general construction in Robert--V\'etois \cite{robertvetoisPROC} reduces the proof of Theorem \ref{th:1} to finding stable critical points to a functional that is the sum of two terms: the first is an explicit local well involving essentially the bubble, the second is the restriction to $\mathcal M$ of a nontrivial global functional $J_0$.

\medskip\noindent Solutions to \eqref{eq:u0} around $u_0$ are all in $\mathcal M$. However, in general, the elements of $\mathcal M$ are not all solutions to \eqref{eq:u0}, that is $\mathcal M$ is not a critical manifold of the problem. Following the terminology of Chapter 2 of the monograph Ambrosetti--Malchiodi \cite{AM}, a critical manifold around $u_0$ is a finite-dimensional manifold ${\mathcal Z}\ni u_0$ of solutions to \eqref{eq:u0}. A critical manifold ${\mathcal Z}$ is nondegenerate if its tangent space at any $u\in{\mathcal Z}$ is exactly $\hbox{Ker}(I_0^{\prime\prime}(u))$, the kernel of the Hessian of $I_0$ at $u$. The existence of a nondegenerate critical manifold around $u_0$ is equivalent to the existence of $\tilde{u}\in C^1(B_1(0)\subset K_0,H_1^2(M))$ such that 
\begin{equation*}%\label{be:cond}
\left\{\begin{array}{l}
\tilde{u}(z)\hbox{ is a solution to \eqref{eq:u0} for all }z\in B_1(0)\subset K_0,\\
\tilde{u}(0)=u_0,\\
K_0=\hbox{Span}\{\partial_{z_i}\tilde{u}(0)/\, i=1,...,d\},\hbox{ where }d:=\hbox{dim}(K_0).\end{array}\right\}\eqno{(BE)}
\end{equation*}
Condition $(BE)$ (for Bianchi--Egnell type condition) is a standard and natural assumption in the finite-dimensional reduction. It is satisfied when $M=\mathbb{R}^n$ and $h\equiv 0$ (see the classical references Rey \cite{Rey} and Bianchi--Egnell \cite{BE}), also for some sign-changing solutions (see the recent example of Musso--Wei \cite{MussoWei}). We refer to Ambrosetti--Malchiodi \cite{AM} for an abstract general setting for the use of nondegenerate critical manifolds.
%The existence of a nondegenerate critical manifold is equivalent to the existence of $\tilde{u}\in C^1(B_1(0)\subset K_0,H_1^2(M))$ such that $\tilde{u}(0)=u_0$, $\tilde{u}(z)$ is a solution to \eqref{eq:u0} for all $z\in B_1(0)\subset K_0$ and 
%\begin{equation}\label{be:cond}
%K_0=\hbox{Span}\{\partial_{z_i}\tilde{u}(0)/\, i=1,...,d\}.
%\end{equation}
%In this case, $\tilde{u}(B_\delta(0))$ is a nondegenerate critical manifold for $\delta>0$ small enough. 

\medskip\noindent In case condition $(BE)$ holds, the manifold $\mathcal M$ is the nondegenerate critical manifold, and minimizing $J_{0|{\mathcal M}}$ exactly amounts to minimizing $I_{0|{\mathcal M}}$, which is a considerable simplification for our problem. However, in general, the Bianchi-Egnell condition $(BE)$ does not hold. It is even exceptional: in Section \ref{sec:ex}, we exhibit examples of degenerate minimizers $u_0$ that are isolated among solutions to \eqref{eq:u0}, and therefore, the only possible critical manifold is $\{u_0\}$ and is degenerate (see Propositions \ref{cond:min} and~\ref{prop:product}). Therefore, the classical methods using nondegenerate critical manifolds (see again the monograph Ambrosetti--Malchiodi \cite{AM}) are ineffective here. We refer to Del Pino--Felmer \cite{delpinofelmer}, Jeanjean--Tanaka \cite{JT}, Byeon--Jeanjean~\cite{BJ}, and Dancer \cite{Dancer JDE2009} for an analysis on $\rn$ without condition $(BE)$ based on topological arguments. 

\medskip\noindent Our aim in the present article is to develop a new method to deal with the absence of nondegenerate critical manifold (that is when the Bianchi-Egnell condition $(BE)$ does not hold) by using analyticity. Indeed, due to our choice of the manifold $\mathcal M$,  we are able to compare precisely all the terms  in the analytic expansions of $I_0$ and $J_0$ on $\mathcal M$. As a consequence, we prove that the restriction of $J_0$ to $\mathcal M$ has a strict local minimum at $u_0$ if and only if $u_0$ is a strict local minimizer of $I_0$ (Theorem \ref{th:minim}). This allows us to get a stable critical point for our problem.

\medskip\noindent This article is organized as follows. In Section \ref{sec:further}, we state byproducts of our analysis. In Section \ref{sec:toolbox}, we define bubbles, we state the general construction theorem via finite-dimensional reduction and we recall existing results. In Section \ref{sec:approx:u0}, we perform a first Lyapunov-Schmidt reduction to construct the analytic manifold $\mathcal M$ of approximations of $u_0$. In Section \ref{sec:reduc:J}, we  reduce the proof of Theorem \ref{th:1} to obtaining a stable well for $J_0$ restricted to $\mathcal M$. In Section \ref{sec:equiv}, we use the analyticity to prove the equivalence of strict local minimization for $I_0$ and $J_0$ on $\mathcal M$. In Section \ref{sec:ex}, we construct examples of degenerate strict local minimizers.

\medskip\noindent{\it Acknowledgement:} the authors thank the referee for careful reading of this manuscript and useful remarks. 

\section{Miscellaneous further results}\label{sec:further}

Theorem \ref{th:1} is a particular case of Theorem \ref{th:2} below:

\begin{thm}\label{th:2} 
Let $(M,g)$ be a compact Riemannian manifold of dimension $n\geq 3$. Let $h\in C^{0,\theta}(M)$ be such that $\Delta_g+h$ is coercive. Assume that there exists $u_0\in C^{2,\theta}(M)$ that is a solution to \eqref{eq:u0} and a strict local minimizer of $I_0$. Assume that one of the following situations holds:
\begin{equation}\label{hyp:th}\left\{\begin{array}{l}
3\leq n\leq 5,\\
n=6\hbox{ and } c_n R_g-h< 2u_0,\\
3\leq n\leq 9 \hbox{ and }h\equiv c_n R_g,\\
n=10 \hbox{ , }h\equiv c_n R_g\hbox{ and }u_0>\frac{5}{567}|\hbox{Weyl}_g|_g^2,\\
n\geq 3,\, (M,g)\hbox{ is locally conformally flat and }h\equiv c_n R_g.
\end{array}\right\}
\end{equation}
Then there exist a solution of type $(u_0-B)$ to \eqref{eq:ue}.
\end{thm}

\noindent We are also in position to construct positive solutions in dimension $n=6$.

\begin{thm}\label{th:3}
Let $(M,g)$ be a smooth compact Riemannian manifold of dimension $n=6$ and let $h\in C^{0,\theta}(M)$ be such that $\Delta_g+h$ is coercive. Assume that there exists $u_0\in C^{2,\theta}(M)$ that is both a solution to \eqref{eq:u0} and an strict local minimizer of $I_0$. Assume that
\begin{equation}\label{hyp:6}
h-c_6R_g>2u_0>0\text{ in }M.
\end{equation}
Then for $\varepsilon>0$ small, equation \eqref{eq:ue} admits a solution $u_\varepsilon>0$ such that $\ue=u_0+B_{\eps}+o(1)$, where $(B_{\eps})_\eps$ is a bubble and $\lim_{\eps\to 0}o(1)=0$ in $H_1^2(M)$.
\end{thm}

\section{Bubbles, general existence theorem and preliminary computations}\label{sec:toolbox}

This section essentially collects existing results from Robert--V\'etois \cites{robertvetois,robertvetoisPROC}.

\subsection{Bubbles}

We follow the terminology in \cites{robertvetoisPROC}. We say that $(B_\eps)_\eps$ is a bubble if there exists $(\xe)_\eps\in M$ and $(\mu_\eps)_\eps\in (0,+\infty)$ such that $\lim_{\eps\to 0}\mu_\eps=0$ and
\begin{equation}\label{def:bubble}
B_\eps(x):=\left(\frac{\sqrt{n(n-2)}\mu_\eps}{\mu_\eps^2+d_g(x,x_\eps)^2}\right)^{\frac{n-2}{2}}\hbox{ for all }x\in M.
\end{equation}
There exists $r_0\in (0, i_g(M))$ and $\varLambda\in C^{\infty}(M\times M)$ such that $(\xi,x)\mapsto \varLambda_\xi(x)>0$, $\varLambda_\xi(\xi)=1$ and :
\begin{enumerate}
\renewcommand{\labelenumi}{(\roman{enumi})}
\item If $(M,g)$ is locally conformally flat (lcf), then $g_\xi=\varLambda_\xi^{4/(n-2)}g$ is flat in $B_\xi(r_0)$.
\item If $(M,g)$ is not locally conformally flat (non lcf) then $g_\xi:=\Lambda_\xi^{\frac{4}{n-2}}g$ satisfies $dv_{g_{\xi}}=(1+O(d_{g_\xi}(\xi,\cdot)^n))\, dx$ in a geodesic normal chart. An immediate consequence is that $R_{g_\xi}(\xi)=|\nabla R_{g_\xi}(\xi)|_{g_\xi}=0$ and $\Delta_{g_\xi}R_{g_\xi}(\xi)=\frac{1}{6}|\hbox{Weyl}_g(\xi)|^2_g$. Moreover, $\nabla\Lambda_\xi(\xi)=0$. This  change of metric is due to Lee--Parker~\cite{leeparker}.
\end{enumerate}
We let $\chi$ be a smooth cutoff function such that $0\le\chi\le1$ in $\mathbb{R}$, $\chi=1$ in $[-r_0/2,r_0/2]$, and $\chi=0$ in $\mathbb{R}\backslash(-r_0,r_0)$. For any $\kappa\in\{-1,1\}$, any positive real number $\delta$ and any point $\xi$ in $M$, we define the function $W_{\kappa,\delta,\xi}$ on $M$ by
$$W_{\kappa,\delta,\xi}(x):=\kappa\chi(d_{g_\xi}(x,\xi))\varLambda_\xi(x)\left(\frac{\sqrt{n(n-2)}\delta}{\delta^2+d_{g_\xi}(x,\xi)^2}\right)^{\frac{n-2}{2}},$$
where $d_{g_\xi}$ is the geodesic distance on $M$ associated with the metric $g_\xi$, the exponential map is taken with respect to the same metric $g_\xi$. As one checks, for any family $(\delta_\eps)_\eps\in (0,+\infty)$ going to $0$ as $\eps\to 0$, there exists a bubble $(B_\eps)_\eps$ such that
\begin{equation}\label{dl:W}
W_{\kappa,\delta_\eps,\xi_\eps}=\kappa B_\eps+o(1)
\end{equation}
in $H_1^2(M)$ when $\eps\to 0$. Bubbles like $W_{\kappa,\delta,\xi}$ with a modification of the metric were introduced by Lee-Parker for an alternate resolution of the Yamabe problem. Using these bubbles smoothly depending on $\xi$ for finite dimensional reduction was first used in the article \cite{EspPisVet} by the second author and his collaborators.

\medskip\noindent{\bf Notations:} Here and in the sequel, $(\Delta_g+h)^{-1}$ denotes the inverse of the natural isometric isomorphism
\begin{equation*}
\begin{array}{cccc}
\Delta_g+h : & H_1^2(M) &\to & (H_1^2(M))^\prime\\
 & \phi &\mapsto & \left(\tau\mapsto \int_M((\nabla\phi,\nabla\tau)_g+h\phi \tau)\, dv_g\right).
\end{array}
\end{equation*}
Any function $f\in L^{\frac{2n}{n+2}}(M)=(L^{\crit}(M))^\prime$ is seen as a linear form on $H_1^2(M)$. In the sequel $C$ will denote a constant independent of $\xi,\delta,\varphi,\eps$. The value of $C$ can change from one line to the other for simplicity.

\subsection{General existence theorem}

For any $\nu_0>0$ and $\eps>0$, we define
$${\mathcal D}_\eps(\nu_0):=\{(\delta,\xi)\in (0,\nu_0)\times M\, /\, |\delta^\eps-1|<\nu_0\}.$$
We define for $\epsilon\in [0,\crit-2)$
$$J_\eps(u):=\frac{1}{2}\int_M\left(|\nabla u|_g^2+hu^2\right)\, dv_g-\frac{1}{\crit-\eps}\int_M |u|^{\crit-\eps}\, dv_g=\frac{1}{2}\Vert u\Vert_h^2-F_\eps(u)$$
for all $u\in H_1^2(M)$, where 
$$\Vert u\Vert_h^2=(u,u)_h=\int_M\left(|\nabla u|_g^2+hu^2\right)\, dv_g\hbox{ and }F_\eps(u):=\frac{1}{\crit-\eps}\int_M H(u)^{\crit-\eps}\, dv_g.$$
Here, $H(u):=|u|$ if $\kappa=-1$ and $H(u):=u_+$ if $\kappa=1$. For any closed subspace $L\subset H_1^2(M)$, $\Pi_L$ will denote the orthogonal projection onto $L$ and $L^\perp$ the orthogonal complement of $L$ with respect to the Hilbert structure $(\cdot,\cdot)_h$.

\medskip\noindent We let $u\in C^1(B_{\nu_0}(0)\subset K_0, H_1^2(M))$ be such that $u(0)=u_0$ and 
\begin{equation}\label{ppty:u:phi}
\left | \det (\Pi_{K_0}\partial_1 u (\varphi),\cdots,\Pi_{K_0}\partial_d u (\varphi)) \right |\geq c_0\prod_{i=1}^d\Vert \partial_i u(\varphi)\Vert_{H_1^2}
\end{equation}
for some $c_0>0$ and all $\varphi\in B_{\nu_0}(0)\subset K_0$. Here, $d:=\dim_{\rr}(K_0)$ and derivatives refer to a fixed basis of $K_0$. %We fix a domain $U\subset M$ such that there exists smooth vector fields $e_1,...,e_n: M\mapsto TM$ such that for all $p\in U$, $\{e_1(p),\dots, e_n(p)\}$ is an orthonormal basis of the tangent space $T_pM$. 
The following existence theorem is a  consequence of Theorem 1.1 in Robert--V\'etois \cite{robertvetoisPROC}:

\begin{thm}\label{th:LS} 
There exists $\nu_0>0$ and there exists $\phi_\eps\in C^1(B_{\nu_0}(0)\times {\mathcal D}_\eps(\nu_0), K_0^\perp)$ such that 
for all $\varphi\in B_{\nu_0}(0)\subset K_0$, $(\delta,\xi)\in {\mathcal D}_\eps(\nu_0)$, the function $u_\eps(\varphi, \delta,\xi):=u(\varphi)+W_{\kappa,\delta,\xi}+\phi_\eps(\varphi,\delta,\xi)$ is a critical point for $J_\eps$ if and only if $(\varphi,\delta,\xi)$ is a critical point of $(\varphi, \delta,\xi)\mapsto J_\eps(u_\eps(\varphi,\delta,\xi))$. Moreover, $\Vert \phi_\eps(\varphi,\delta,\xi)\Vert_h\leq C\cdot R_\eps(\varphi,\delta, \xi)$ where
\begin{equation}\label{def:R}
R_\eps(\varphi, \delta,\xi):=\Vert \Pi_{K_{\delta,\xi}^\perp}\left(u(\varphi)+W_{\kappa,\delta,\xi}-(\Delta_g+h)^{-1}(F_\eps^\prime (u(\varphi)+W_{\kappa,\delta,\xi}))\right)\Vert_h\,.
\end{equation}
The space $K_{\delta,\xi}$ is defined below. 
\end{thm}

\noindent The projection onto $K_{\delta,\xi}^\perp$ in the rest $R_\eps(\varphi, \delta,\xi)$ follows from Subsection 5.3 in \cite{robertvetoisPROC}. The function $\phi_\eps$ is defined implicitely as follows: given $(\varphi,\delta,\xi)\in B_{\nu_0}(0)\times {\mathcal D}_\eps(\nu_0)$, $\phi_\eps(\varphi,\delta,\xi)$ is the sole element of $K_{\delta,\xi}^\perp$ such that
$$\Pi_{K_{\delta,\xi}^\perp}\left(u_\eps(\varphi,\delta,\xi)-(\Delta_g+h)^{-1}(F_\eps^\prime (u_\eps(\varphi,\delta,\xi))\right)=0\,.$$
%In addition, it is not possible to replace $\phi(\varphi,\delta,\xi)$ by another element of $K_0^\perp$ close to $0$ in this expression. 
The linear space $K_{\delta,\xi}$ is defined as
$$K_{\delta,\xi}:=\hbox{Span}\left\{\varphi\, ,\, Z_{\delta,\xi}\, ,\,  Z_{\delta,\xi,X} \, ,\, \varphi\in  K_0\hbox{ and }X\in T_\xi M\right\},$$
where
\begin{align*}
&Z_{\delta,\xi}(x):=\chi(d_{g_\xi}(x,\xi))\Lambda_\xi(x)\delta^{\frac{n-2}{2}}\frac{d_{g_\xi}(x,\xi)^2-\delta^2}{(\delta^2+d_{g_\xi}(x,\xi)^2)^{\frac{n}{2}}}\,,\\%\label{def:Z}
&Z_{\delta,\xi,X}(x):=\chi(d_{g_\xi}(x,\xi))\Lambda_\xi(x)\delta^{\frac{n}{2}}\frac{\langle (\hbox{exp}_\xi^{g_\xi})^{-1}(x),X\rangle_{g_\xi(\xi)}}{(\delta^2+d_{g_\xi}(x,\xi)^2)^{\frac{n}{2}}}%\label{def:Z:omega}
\end{align*}
for all $x\in M$.

\subsection{Estimate of the error term}

%\medskip\noindent In the sequel, we define
%$$F_\eps(u):=\frac{1}{\crit-\eps}\int_M |u|^{\crit-\eps}\, dv_g$$
%for all $u\in H_1^2(M)$. 
For simplicity, we will often write $W:=W_{\kappa,\delta,\xi}$ and $\phi:=\phi_\eps(\varphi,\delta, \xi)$ in this section. 
%Since $F_\eps\in C^{2,\theta}(H_1^2(M))$ for all $\theta\in (0,\min \{\crit-2,1\})$ with associated norm uniformly bounded on any bounded domain with respect to $\epsilon$ small, a Taylor expansion yields
%\begin{eqnarray}
%F_\eps(u(\varphi,\delta, \xi))&=& F_\eps(u(\varphi)+W+\phi))\\
%&=& F_\eps'(u(\varphi))+F_\eps'(W)+F_\eps^{\prime\prime}(u(\varphi)+W)\phi\\
%&& +O\left(\Vert \phi\Vert_h^{1+\theta}\right)+O\left(\Vert F_\eps'(u(\varphi)+W)-F_\eps'(u(\varphi)-F_\eps'(W)\Vert_{H_1^2(M)'}\right)
%\end{eqnarray}
It follows from \cite{robertvetois}*{Sections 5 and~7}, that
\begin{equation}\label{ineq:F:prime}
\Vert F_\eps'(u(\varphi)+W)-F_\eps'(u(\varphi))-F_\eps'(W)\Vert_{H_1^2(M)'}\leq C\cdot \eps_1(\delta),
\end{equation}
\begin{equation}\label{ineq:F}
 F_\eps(u(\varphi)+W)-F_\eps(u(\varphi))-F_\eps(W)-F_\eps'(u(\varphi))W-F_\eps'(W)u(\varphi)=O\left(\eps_2(\delta)\right),
\end{equation}
and
\begin{multline}
\Vert W-(\Delta_g+h)^{-1}(F_\eps^\prime(W))\Vert_h\label{ineq:W:error}\\
\leq C\cdot \left(\eps\ln\frac{1}{\delta}+\eps_1(\delta) +{\bf 1}_{\left\{n\geq 7\right\}}\Vert h-c_n R_g\Vert_\infty\delta^2 +{\bf 1}_{\left\{n\geq 15\hbox{ and non lcf}\right\}}\delta^4\right),
\end{multline}
where
\begin{equation}\label{def:eps:2}
\eps_1(\delta):=\left\{\begin{array}{ll}
\delta^{\frac{n-2}{2}} & \hbox{ if }n<6\\
\delta^2\left(\ln \frac{1}{\delta}\right)^{\frac{2}{3}}& \hbox{ if }n=6\\
\delta^{\frac{n+2}{4}} & \hbox{ if }n>6
\end{array}\right\}
\text{ and }\eps_2(\delta):=\left\{\begin{array}{ll}
\delta & \hbox{ if }n=3\\
\delta^2\ln \frac{1}{\delta}& \hbox{ if }n=4\\
\delta^{\frac{n}{2}} & \hbox{ if }n\geq 5
\end{array}\right\}.
\end{equation}
Plugging \eqref{ineq:F:prime} and \eqref{ineq:W:error} in \eqref{def:R} yields
\begin{eqnarray}
%&&\Vert \Pi_{K_{\delta,\xi}^\perp}\left(u(\varphi)+W_{\kappa,\delta,\xi}-(\Delta_g+h)^{-1}((u(\varphi)+W_{\kappa,\delta,\xi})^{\crit-1-\eps})\right)\Vert_h\\
&&R_\eps(\varphi,\delta,\xi)\leq  C\cdot \Vert \Pi_{K_{\delta,\xi}^\perp}\left(u(\varphi)-(\Delta_g+h)^{-1}(F_0'(u(\varphi))\right)\Vert_h\label{bnd:R:1}\\
&&+O\left(\eps\ln\frac{1}{\delta}+\eps_1(\delta) +{\bf 1}_{\left\{n\geq 7\right\}}\Vert h-c_n R_g\Vert_\infty\delta^2 +{\bf 1}_{\left\{n\geq 15\hbox{ and non lcf}\right\}}\delta^4\right)\nonumber.
\end{eqnarray}
%It then follows from XXX that
%\begin{eqnarray}
%&&\Pi_{K_{\delta,\xi}^\perp}\left(\phi-(\Delta_g+h)^{-1}(F_\eps'(u(\varphi)+W)\phi)\right) =\Pi_{K_{\delta,\xi}^\perp}\left(u(\varphi)-(\Delta_g+h)^{-1}(F_0'(u(\varphi))\right)\\
%&&+O\left(\Vert \phi\Vert_h^{1+\theta}\right)+O\left(\eps\ln\frac{1}{\delta}+\eps(\delta) +1_{n\geq 7}\Vert h-c_n R_g\Vert_\infty\delta^2 +1_{n\geq 15\hbox{ and }non\, lcf}\delta^4\right)
%\end{eqnarray}
%Since XXX holds and $L$ is invertible around $0$ (see XXX), we then get that\footnote{aller plus vite car on a deja le controle de la norme de phi dans le theoreme}
%\begin{eqnarray}
%&&\Vert \phi\Vert_h\leq C\cdot \left(\eps\ln\frac{1}{\delta}+\eps(\delta) +1_{n\geq 7}\Vert h-c_n R_g\Vert_\infty\delta^2 +1_{n\geq 15\hbox{ and }non\, lcf}\delta^4\right)
%\end{eqnarray}

\subsection{First expansion of the energy $J_\eps$}

The Taylor expansion of $J_\eps$, the control of $\phi_\eps$ in Theorem \ref{th:LS} and the definition \eqref{def:R} of $R_\eps(\varphi,\delta,\xi)$ yield
\begin{align*}
&J_\eps(u(\varphi)+W+\phi)\\
=& J_\eps(u(\varphi)+W)+(u(\varphi)+W-(\Delta_g+h)^{-1}(F_\eps^\prime(u(\varphi)+W)), \phi)_h+O(\Vert\phi\Vert_h^2)\\
=&J_\eps(u(\varphi)+W)+(\Pi_{K_{\delta,\xi}^\perp}(u(\varphi)+W-(\Delta_g+h)^{-1}(F_\eps^\prime(u(\varphi)+W))), \phi)_h+O(\Vert\phi\Vert_h^2)\\
=& J_\eps(u(\varphi)+W)+O(R_\eps(\varphi,\delta,\xi)^2).
\end{align*}
It then follows from \eqref{ineq:F}  and \eqref{def:eps:2} that
\begin{multline}
J_\eps(u(\varphi)+W+\phi)= J_\eps(u(\varphi)) + J_\eps(W_{\kappa,\delta,\xi})\label{eq:J:1}\\
+\big(u(\varphi)-(\Delta_g+h)^{-1}(F_\eps'(u(\varphi))),W\big)_h
-F_\eps'(W)u(\varphi)+O\left(R_\eps(\varphi,\delta,\xi)^2+\eps_2(\delta)\right).
\end{multline}
Since $\varphi\mapsto u(\varphi)>0$ is $C^1$, $u(0)=u_0$ is a solution to \eqref{eq:u0}, we get that
\begin{equation}\label{eq:W:1}
\big(u(\varphi)-(\Delta_g+h)^{-1}(u(\varphi)^{\crit-1-\eps}),W\big)_h= f_1(\varphi,\xi)\delta^{\frac{n-2}{2}}+o(\delta^{\frac{n-2}{2}})
\end{equation}
when $\delta,\eps\to 0$ and $f_1\in C^1(B_{\nu_0}(0)\times M, \rr)$ ($B_{\nu_0}(0)\subset K_0$) and $f_1(0,\xi)=0$ for all $\xi\in M$. It follows from \cite{robertvetois} that
\begin{equation}\label{eq:W:2}
F_\eps'(W)u(\varphi)= \frac{\kappa 2^n\omega_{n-1}K_n^{-n}}{n(n(n-2))^{\frac{n-2}{4}}\omega_n}u(\varphi)[\xi]\delta^{\frac{n-2}{2}}+O(\delta^{\frac{n-2}{2}}(o(1)+|\delta^\eps-1|))
%F_\eps'(W_{\kappa,\delta,\xi})u(\varphi)=\kappa \left(\int_{\rn}U_0^{\crit-1}\, dx\right)u(\varphi)[\xi]\delta^{\frac{n-2}{2}}+O(\delta^{\frac{n-2}{2}}(o(1)+|\delta^\eps-1|))
\end{equation}
when $(\delta,\eps)\to 0$. Here, $\omega_k$ is the volume of the canonical unit $k-$sphere in $\rr^{k+1}$ and $K_n$ is the best constant of the Sobolev inequality $\Vert u\Vert_{\crit}\leq K\Vert\nabla u\Vert_2$ in $\rn$. %where $U_0(x):=\left(\sqrt{n(n-2)}/(1+|x|^2)\right)^{\frac{n-2}{2}}$ for all $x\in \rn$. 
Finally, expanding $J_\eps(u(\varphi))$ with respect to $\eps$ and collecting  \eqref{eq:J:1}, \eqref{eq:W:1} and \eqref{eq:W:2} yield
\begin{multline}
J_\eps(u(\varphi)+W+\phi)= J_0(u(\varphi)) + \eps f_2(\varphi)+J_\eps(W_{\kappa,\delta,\xi}) \label{eq:J:2}\\
+\left(f_1(\varphi,\xi)-\frac{\kappa 2^n\omega_{n-1}K_n^{-n}}{n(n(n-2))^{\frac{n-2}{4}}\omega_n}u(\varphi)[\xi]\right)\delta^{\frac{n-2}{2}}\\
+O\left(R_\eps(\varphi,\delta,\xi)^2+\eps_2(\delta)+\delta^{\frac{n-2}{2}}(o(1)+|\delta^\eps-1|)\right) +o(\eps)
\end{multline}
when $\delta,\eps\to 0$. Here, $f_2\in C^1(B_{\nu_0}(0)\subset K_0, \rr)$

\subsection{Expansion of $J_\eps(W_{\kappa,\delta,\xi})$}

The following result was obtained in \cite{robertvetois}: there exists $\beta_n>0$ such that
\begin{multline}
J_\eps(W_{\kappa,\delta,\xi})=\frac{K_n^{-n}}{n}\left(1-\beta_n\eps-\frac{(n-2)^2}{4}(\delta^\eps-1)\right)+O\left(\eps\delta^2+\eps^2+(\delta^\eps-1)^2\right)\label{exp:J:W}\\
+O\left({\bf 1}_{\left\{n\leq 5\hbox{ or lcf}\right\}}\delta^{n-2}\right)\\
+\frac{K_n^{-n}}{n}\left\{
\begin{array}{ll}
O\left(\Vert h-c_3 R_g\Vert_{C^{0,\theta}}\delta\right)&\hbox{ if }n=3\\
&\\
3(h-c_4 R_g)(\xi)\delta^2\ln\frac{1}{\delta}+O\left(\Vert h-c_4 R_g\Vert_{C^{0,\theta}}\delta^{2}\right)&\hbox{ if }n=4\\
&\\
\frac{2(n-1)}{(n-2)(n-4)}(h-c_n R_g)(\xi)\delta^2+O\left(\Vert h-c_n R_g\Vert_{C^{0,\theta}}\delta^{2+\theta}\right)&\hbox{ if }n\geq 5
\end{array}\right\}\\
+\frac{K_n^{-n}}{n}\left\{
\begin{array}{ll}
-\frac{1}{64}|\hbox{Weyl}_g(\xi)|_g^2\delta^4\ln\frac{1}{\delta}+O\left(\delta^{4}\right)&\hbox{ if }n=6\hbox{ and non lcf}\\
&\\
-\frac{1}{24(n-4)(n-6)}|\hbox{Weyl}_g(\xi)|_g^2\delta^4+O(\delta^5)&\hbox{ if }n\geq 7\hbox{ and non lcf}
\end{array}\right\}.
\end{multline}

\section{Suitable approximation of $u_0$ and analyticity}\label{sec:approx:u0}

In \cite{robertvetois}, the blowing-up solutions of type $(u_0-B)$ are directly modeled on a nondegenerate function $u_0$. When $u_0$ is degenerate, the kernel $K_0$ plays a role in the finite-dimensional reduction and we consider a manifold of functions around $u_0$ parametrized locally by $K_0$.

\begin{prop}\label{prop:def:phibar}
There exist $\nu_0>0$ small and $\phi\in C^1(B_{\nu_0}(0)\subset K_0, K_0^\perp)$ such that for all $\varphi\in K_0$ and $\psi\in K_0^{\perp}$ satisfying $\Vert\varphi\Vert_h,\Vert\psi\Vert_h<\nu_0$, we have that
$$\Pi_{K_0^\perp}(u_0+\varphi+\psi-(\Delta_g+h)^{-1}(F_0'(u_0+\varphi+\psi)))=0\,\Leftrightarrow\, \psi=\phi(\varphi).$$
In particular, $\phi$ vanishes up to order $1$ at $0$. Moreover, taking $\nu_0$ smaller if necessary, $u_0+\varphi+\phi(\varphi)\in C^{2,\theta}(M)$ is positive for all $\varphi\in B_{\nu_0}(0)$ and $\phi: B_{\nu_0}(0)\to C^{2,\theta}(M)$ is analytic with respect to the associated topologies.
\end{prop}

\noindent The analytic manifold of approximation is ${\mathcal M}:=\{u_0+\varphi+\phi(\varphi)/\, \varphi\in B_{\nu_0}(0)\subset K_0\}$.

\medskip\noindent Proposition \ref{prop:def:phibar} is a particular case of a more general result. Some definitions and notations are required in order to state the general result. We fix $f\in C^1(\rr)$ and we assume that there exists $u_0\in C^{2,\theta}(M)$ such that 
\begin{equation}\label{eq:u0:f}
\Delta_g u_0+h u_0=f(u_0)\hbox{ in }M.
\end{equation}
We define 
\begin{equation}\label{def:K0:f}
K_0:=\{\varphi\in C^{2,\theta}(M)/\, \Delta_g \varphi+h\varphi=f'(u_0)\varphi\}.
\end{equation}
In the sequel, $K_0$ will be regarded as a subset of the Hilbert space $H_1^2(M)$. It follows from Fredholm's theory for Hilbert spaces that $K_0$ is of finite dimension $d\in\mathbb{N}$. We prove the following result in the spirit of Dancer \cites{DancerMA2003}:

\begin{prop}\label{prop:analycity} 
We let $f\in C^1(\rr)$ and $u_0\in C^{2,\theta}(M)$ be a solution to \eqref{eq:u0:f}. We let $K_0$ be as in  \eqref{def:K0:f}.Then there exist $\nu>0$ and $\phi\in C^1(B_{\nu}(0)\subset K_0, K_0^\perp\cap C^{2,\theta}(M))$ such that for all $\varphi\in B_{\nu}(0)\subset K_0$ and $\psi\in B_{\nu}(0)\subset K_0^\perp$,
\begin{equation}
\label{def:psi}
\Pi_{K_0^\perp}(u_0+\varphi+\psi-(\Delta_g+h)^{-1}(f(u_0+\varphi+\psi)))=0\,\Leftrightarrow \, \psi=\phi(\varphi).
\end{equation}
Moreover, if $f$ is analytic on an open interval $I$ and $u_0(x)\in I$ for all $x\in M$, then $\phi$ is analytic around $0$.
\end{prop}

\noindent As one checks, the function $x\mapsto |x|^{\crit-2}x$ is $C^1$ on $\rr$ and analytic on $(0,+\infty)$. Therefore Proposition \ref{prop:def:phibar} is a direct consequence of Proposition \ref{prop:analycity}.
\begin{proof}[Proof of Proposition \ref{prop:analycity}] The first part of the statement is a direct application of the implicit function theorem and regularity theory.
Since $M$ is compact and $u_0$ is continuous, it follows from the analyticity of $f$ that there exists $A,B>0$ such that
\begin{equation}\label{bnd:a:k}
\left |a_k(u_0(x))\right |\leq A\cdot B^k\hbox{ for all }k\geq 0\hbox{ and }x\in M,
\end{equation}
where 
$$f(u_0(x)+h)=\sum_{k=0}^\infty a_k(u_0(x))h^k\hbox{ for all }x\in M\hbox{ and }h\in (-B^{-1},B^{-1}).$$
Since $\phi$ is $C^\infty$ its differential vanishes at $0$, we write for any $L\geq 2$ that
$$\phi(\varphi)=\sum_{l=2}^LP_l(\varphi)+o(\Vert\varphi\Vert^L)\hbox{ when }\varphi\to 0\,,$$
where for all $l\geq 2$ and $\varphi\in B_{\nu}(0)\subset K_0$, $P_l(\varphi)\in K_0^\perp$ is a homogeneous polynomial of degree $l$. We set $P_1(\varphi):=\varphi\in K_0$. Therefore, for any $L\geq 1$, we have that 
$$f(u_0+\varphi+\phi(\varphi))=\sum_{k=0}^La_k(u_0)\left(\sum_{l=1}^LP_l(\varphi)\right)^k+o(\Vert\varphi\Vert^L)$$
when $\varphi\to 0$. We write that
\begin{equation}
\label{poly:expo:k}
\left(\sum_{i=1}^L X_i\right)^k=\sum_{j=0}^\infty Q_{k,L,j}(X_1,...,X_L),
\end{equation}
where
$$Q_{k,L,j}(X_1,...,X_L):=\sum_{\sum_1^L r_l=k\,;\, \sum_{1}^L lr_l=j}\frac{k!}{\prod_{l=1}^L r_l!}\prod_{l=1}^LX_l^{r_l}.$$
Note that $Q_{k,L,j}(X_1,...,X_L)=0$ when $j\not\in [k,Lk]$, so all the sums make sense. Therefore, for any $L\geq 2$, the term of degree $L$ in \eqref{def:psi} is
$$\Pi_{K_0^\perp}\left(P_L(\varphi)-(\Delta_g+h)^{-1}\left(\sum_{k=0}^La_k(u_0)Q_{k,L,L}(P_1(\varphi),...,P_L(\varphi))\right)\right)=0$$
for all $L\geq 2$. In the sum, the term for $k=0$ is $0$, and the term for $k=1$ is $a_1(u_0)P_L(\varphi)=f'(u_0)P_L(\varphi)$. Therefore, we have that
\begin{equation}\label{carac:P:L}
P_L(\varphi)=L_0^{-1}\Pi_{K_0^\perp}\left((\Delta_g+h)^{-1}\left(\sum_{k=2}^La_k(u_0)Q_{k,L,L}(P_1(\varphi),...,P_L(\varphi))\right)\right)
\end{equation}
for all $L\geq 2$, where $L_0: K_0^\perp\to K_0^\perp$ is the isomorphism given by
$$L_0(\psi)=\Pi_{K_0^\perp}\left(\psi-(\Delta_g+h)^{-1}\left(f'(u_0)\psi\right)\right)\hbox{ for all }\psi\in K_0^\perp.$$
Note that since $k,L\geq 2$, the right-hand side of \eqref{carac:P:L} is independent of $P_L(\varphi)$. We fix $\alpha\in (0,1)$. It follows from elliptic theory that there exists $C>0$ depending on $(M,g)$, $h$ and $f'(u_0)$ such that 
\begin{equation}\label{bnd:P:L:C1}
\Vert P_L(\varphi)\Vert_{C^{1,\alpha}}\leq C\Vert \sum_{k=2}^La_k(u_0)Q_{k,L,L}(P_1(\varphi),...,P_L(\varphi))\Vert_\infty
\end{equation}
for all $L\geq 2$. We fix $K\geq 2$. Summing \eqref{bnd:P:L:C1} from $L=2$ to $K$, using \eqref{bnd:a:k}, \eqref{poly:expo:k} and the nonnegativity of the coefficients of $Q_{k,L,L}$, we get that
\begin{align}
\sum_{L=2}^K \Vert P_L(\varphi)\Vert_{C^{1,\alpha}}&\leq C\cdot A\sum_{k=2}^K\sum_{L=k}^K B^k Q_{k,L,L}(\Vert P_1(\varphi)\Vert_\infty,...,\Vert P_L(\varphi)\Vert_\infty)\label{bnd:sum:P:L}\\
&\leq C\cdot A\sum_{k=2}^K\sum_{L=k}^K B^k Q_{k,K,L}(\Vert P_1(\varphi)\Vert_\infty,...,\Vert P_K(\varphi)\Vert_\infty)\nonumber\\
&\leq C\cdot A\sum_{k=2}^K \left(B\sum_{l=1}^K\Vert P_l(\varphi)\Vert_\infty\right)^k.\nonumber
\end{align}
We define
$$h_K(t):=\sup_{\Vert\varphi\Vert_\infty\leq t}\sum_{L=2}^K \Vert P_L(\varphi)\Vert_{\infty}\,.$$
It follows from \eqref{bnd:sum:P:L} that
$$t+h_K(t)\leq \frac{1}{2B}\;\Rightarrow \; h_K(t)\leq 2 C\cdot A\cdot B^2\cdot\left(t+h_K(t)\right)^2.$$
Therefore, since $h_K$ is continuous and non-decreasing, we get that
$$t<\eps_0:=\min\left(\frac{1}{4B},\frac{1}{16 AB^2C}\right) \; \Rightarrow \; h_K(t)\leq \eps_0\hbox{ for all }K\geq 2\,.$$
As a consequence, the series $(\sum_{L=2}^\infty P_L(\varphi))$ converges uniformly on $B_{\eps_0/2}(0)\subset K_0$ in the $C^{0,\alpha}$--norm. Inequality \eqref{bnd:sum:P:L} yields the convergence in $C^{1,\alpha}(M)$. The characterization \eqref{def:psi} then yields
$$\phi(\varphi)=\sum_{l=2}^\infty P_l(\varphi)\hbox{ for all }\varphi\in B_{\eps_0}(0)\subset K_0\,.$$
Elliptic theory yields convergence in $C^{2,\theta}(M)$. This proves analyticity. 
\end{proof}

\section{Reduction of the problem to the analysis of $J_0(u_0+\varphi+\phi(\varphi))$}\label{sec:reduc:J}

From now on, we define:
$$u(\varphi):=u_0+\varphi+\phi(\varphi)$$
for all $\varphi\in B_{\nu_0}(0)\subset K_0$, where $\phi(\varphi)$ is defined in Proposition \ref{prop:def:phibar}. In particular, \begin{equation}\label{eq:phi}
\Pi_{K_0^\perp}\left(u(\varphi)-(\Delta_g+h)^{-1}(F_0'(u(\varphi))\right)=0
\end{equation}
for all $\varphi\in B_{\nu_0}(0)\subset K_0$. Since $d\phi_0\equiv 0$, it then follows from Proposition \ref{prop:def:phibar} that $u$ satisfies the hypothesis \eqref{ppty:u:phi}. For $0<a<b$ to be fixed later, we define
$$\delta:=t \eps^{\frac{2}{n-2}}$$
for $t\in [a,b]$. We assume that
$$\{3\leq n\leq 6\}\hbox{ or }\{h\equiv c_nR_g\hbox{ and }3\leq n\leq 10\}\hbox{ or }\{h\equiv c_nR_g\hbox{ and lcf}\}.$$
Taking into account the expressions \eqref{def:eps:2}, \eqref{bnd:R:1}, \eqref{eq:J:2}, \eqref{exp:J:W},  and \eqref{eq:phi}, we then get that
\begin{multline}
J_\eps(u(\varphi)+W+\phi)=J_0(u(\varphi)) + \eps f_2(\varphi)+\frac{K_n^{-n}}{n}\left(1-\beta_n\eps+\frac{n-2}{2}\eps\ln\frac{1}{\eps}\right)\label{eq:J:3}\\
+\eps\cdot\frac{K_n^{-n}}{n}\cdot\left(\frac{(n-2)^2}{4}\ln\frac{1}{t}+F(\varphi, \xi)t^{\frac{n-2}{2}}\right)+o(\eps)
\end{multline}
when $\eps\to 0$ uniformly with respect to $t\in [a,b]$. Here, $F\in C^1(B_{\nu_0}(0)\times M)$ and we have that
\begin{multline*}
F(0,\xi)=-\kappa\frac{2^n\omega_{n-1}u_0(\xi)}{(n(n-2))^{\frac{n-2}{4}}\omega_n}\\
+\left\{\begin{array}{ll}
\frac{2(n-1)}{(n-2)(n-4)}(h-c_n R_g)(\xi)&\hbox{ if }n=6\\
-\frac{1}{24(n-4)(n-6)}|\hbox{Weyl}_g(\xi)|^2&\hbox{ if }n=10\hbox{ and }h\equiv c_n R_g\\
0 &\hbox{ otherwise.}
\end{array}\right.
\end{multline*}
The assumptions \eqref{hyp:th} (for $\kappa=-1$) and \eqref{hyp:6} (for $\kappa=1$) then yield 
$$F(0,\xi)>0\hbox{ for all }\xi\in M.$$
We define 
$$a:=\frac{1}{2}\left(\frac{n-2}{2\min_{\xi\in M}F(0,\xi)}\right)^{\frac{2}{n-2}}\hbox{ and }b:=2\left(\frac{n-2}{2\min_{\xi\in M}F(0,\xi)}\right)^{\frac{2}{n-2}}.$$
Since $u_0$ is a strict local minimizer of $I_0$, it follows from Theorem \ref{th:minim} of next section that there exists $\nu_1\in (0, \nu_0/2)$ such that
\begin{equation}\label{hyp:J:min}
J_0(u(\varphi))>J_0(u_0)\hbox{ for all }\varphi\in B_{2\nu_1}(0)\setminus \{0\}.
\end{equation}
Due to compactness, for any $\eps>0$, there exists $(\varphi_\eps, t_\eps, \xi_\eps)\in  \overline{B}_{\nu_1}(0)\times[a,b]\times M$ such that
\begin{multline*}
\min_{(\varphi, t, \xi)\in \overline{B}_{\nu_1}(0)\times[a,b]\times  M}J_\eps(u(\varphi)+W_{\kappa,t\eps^{\frac{2}{n-2}},\xi}+\phi_\eps(\varphi, t\eps^{\frac{2}{n-2}},\xi))\\
=J_\eps(u(\varphi_\eps)+W_{\kappa,t_\eps\eps^{\frac{2}{n-2}},\xi_\eps}+\phi_\eps(\varphi_\eps, t_\eps\eps^{\frac{2}{n-2}},\xi_\eps)).
\end{multline*}
It then follows from the Taylor expansion \eqref{eq:J:3}, the choice of $0<a<b$ and \eqref{hyp:J:min} that $t_\eps\in (a,b)$ and $\varphi_\eps\in B_{\nu_1}(0)$ for small $\eps>0$. Moreover, we have that
$$\lim_{\eps\to 0}t_\eps=\left(\frac{n-2}{2\min_{\xi\in M}F(0,\xi)}\right)^{\frac{2}{n-2}}\hbox{ and }\lim_{\eps\to 0}\varphi_\eps=0\,,$$
and $(\xi_\eps)_{\eps>0}$ approaches the set of minimizers of $F(0,\cdot)$ when $\eps>0$ is small. Therefore, since $(\varphi_\eps, t_\eps, \xi_\eps)$ lies in the interior of the domain, it is a critical point for the minimizing functional, and therefore, $(\varphi_\eps, t_\eps\eps^{\frac{2}{n-2}}, \xi_\eps)$ is a critical point for
$$(\varphi,\delta,\xi)\mapsto J_\eps(u(\varphi)+W_{\kappa,\delta,\xi}+\phi_\eps(\varphi, \delta,\xi)).$$
It then follows from Theorem \ref{th:LS} that $u_\eps:=u(\varphi_\eps)+W_{\kappa,t_\eps\eps^{\frac{2}{n-2}}, \xi_\eps}+\phi_\eps(\varphi_\eps, t_\eps\eps^{\frac{2}{n-2}}, \xi_\eps)$ is a solution to
$$\Delta_g u_\eps+hu_\eps=|u_\eps|^{\crit-2-\eps}u_\eps\hbox{ in }M$$
for $\eps>0$ small, and in addition, due to \eqref{dl:W} and the error control of $\phi_\eps$ in Theorem \ref{th:LS}, we have that
$$u_\eps=u_0+\kappa B_\eps+o(1)$$
in $H_1^2(M)$ when $\eps\to 0$, where $B_\eps$ is as in \eqref{def:bubble} with $\mu_\eps:=t_\eps\eps^{\frac{2}{n-2}}$. This proves Theorems \ref{th:2} and \ref{th:3}, and therefore Theorem \ref{th:1}. 

\medskip\noindent We are now left with proving Theorem \ref{th:minim}.

\section{Equivalence of strict local minimizers}\label{sec:equiv}

This section is devoted to the proof of the following:

\begin{thm}\label{th:minim} 
The function $u_0$ is a strict local minimizer of $I_0$ iff $0$ is a strict local minimizer of $\varphi\mapsto J_0(u_0+\varphi+\phi(\varphi))$. 
\end{thm}

\noindent The proof goes through four claims and uses the analyticity of $\varphi\mapsto \phi(\varphi)$.

\begin{claim} \label{claim:1} 
There exists $\nu_0>0$ such that
$$\Vert u_0+\varphi+\phi(\varphi)\Vert_{h}^{2}-\Vert u_0+\varphi+\phi(\varphi)\Vert_{\crit}^{\crit}=\sum_{L=3}^\infty A_L(\varphi)$$
and
$$\Vert u_0+\varphi+\phi(\varphi)\Vert_{\crit}^{\crit} = \Vert u_0\Vert_{\crit}^{\crit}-\frac{n}{2}\sum_{L=3}^\infty\frac{L-2}{L}A_L(\varphi)$$
for $\varphi\in B_{\nu_0}(0)\subset K_0$, where for any $L\geq 3$, $A_L(\varphi)$ is a homogeneous polynomial of degree $L$.
\end{claim}

\proof[Proof of Claim \ref{claim:1}] 
We are going to compute the Taylor expansions of the two left-hand-sides and we will use the analyticity of $\varphi\mapsto\phi(\varphi)$ to prove Claim \ref{claim:1}. We fix $N\geq 2$. It follows from \eqref{poly:expo:k} that
\begin{multline}
\Vert u_0+\varphi+\phi(\varphi)\Vert_{\crit}^{\crit} = \int_M \left(u_0+\sum_{l=1}^NP_l(\varphi)\right)^{\crit}\, dv_g+o\left(\Vert \varphi\Vert^N\right)= \Vert u_0\Vert_{\crit}^{\crit}+\\
\sum_{L=1}^N\sum_{j=1}^L\sum_{\sum_{l=1}^L r_l=j\, ;\, \sum_{l=1}^L lr_l=L}\frac{\prod_{i=0}^{j-1}(\crit-i)}{\prod_{l=1}^L r_l!}\int_M u_0^{\crit-j}\prod_{l=1}^L P_l(\varphi)^{r_l}\, dv_g
+o\left(\Vert \varphi\Vert^N\right).\label{eq:37}
\end{multline}
We claim that \begin{equation}\label{ortho:u0:K0}
u_0\in K_0^\perp.
\end{equation}
We prove the claim. We let $\varphi$ be in $K_0$. The self-adjointness of the Laplacian yields
$$(u_0,\varphi)_h=\int_M(\Delta_g u_0+hu_0)\varphi\, dv_g=\int_M(\Delta_g \varphi+h\varphi)u_0\, dv_g.$$
It then follows from equation \eqref{eq:u0} and the definition \eqref{def:K0} of $K_0$ that $(u_0,\varphi)_h=0$. This proves the claim. 

\medskip\noindent It follows from \eqref{ortho:u0:K0} that the term for $L=1$ in \eqref{eq:37} is $\crit\int_M u_0^{\crit-1}\varphi\, dv_g=0$.  Separating the cases $j=1$ and $j\geq 2$, we get that
\begin{multline}
\Vert u_0+\varphi+\phi(\varphi)\Vert_{\crit}^{\crit}\\
= \Vert u_0\Vert_{\crit}^{\crit}+\sum_{L=2}^N\sum_{j=2}^L\sum_{\sum_l r_l=j\, ;\, \sum_l lr_l=L}\frac{\prod_{i=0}^{j-1}(\crit-i)}{\prod_{l=1}^L r_l!}\int_M u_0^{\crit-j}\prod_{l=1}^L P_l(\varphi)^{r_l}\, dv_g\\
+\crit\sum_{L=2}^N\int_M u_0^{\crit-1}P_L(\varphi)\, dv_g+o\left(\Vert \varphi\Vert^N\right).\label{eq:1}
\end{multline}
For $L\geq 2$, it follows from the expression \eqref{carac:P:L} of $P_L(\varphi)$ that 
\begin{equation}\label{eq:3}
(L_0 P_L(\varphi), u_0)_h=\sum_{j=2}^L\sum_{\sum_l r_l=j\, ;\, \sum_l lr_l=L}\frac{\prod_{i=1}^{j}(\crit-i)}{\prod_{l=1}^L r_l!}\int_M u_0^{\crit-j}\prod_{l=1}^L P_l(\varphi)^{r_l}\, dv_g.
\end{equation}
Since the operator $L_0$ is symmetric, we have that
\begin{equation}\label{eq:2}
(L_0 P_L(\varphi), u_0)_h=( P_L(\varphi), L_0u_0)_h=-(\crit-2)\int_M u_0^{\crit-1}P_L(\varphi)\, dv_g.
\end{equation}
Plugging into \eqref{eq:1} the expression of $\int_M u_0^{\crit-1}P_L(\varphi)\, dv_g$ obtained by combining \eqref{eq:2} and \eqref{eq:3}, we get that
\begin{multline}
\Vert u_0+\varphi+\phi(\varphi)\Vert_{\crit}^{\crit}= \Vert u_0\Vert_{\crit}^{\crit}+\\\frac{n}{2}\sum_{L=2}^N\sum_{j=2}^L\sum_{\sum_l r_l=j;\sum_l lr_l=L}\frac{(j-2)\prod_{i=1}^{j-1}(\crit-i)}{\prod_{l=1}^L r_l!}\int_M u_0^{\crit-j}\prod_{l=1}^L P_l(\varphi)^{r_l} dv_g
+o\left(\Vert \varphi\Vert^N\right).\label{eq:star}
\end{multline}
Note that the term in the above sum vanishes for $j=2$. As one checks, for any $3\leq j\leq L$, we have that
%, we get that
%\begin{eqnarray}
%&&\Vert u_0+\varphi+\phi(\varphi)\Vert_{\crit}^{\crit} =\Vert u_0\Vert_{\crit}^{\crit}\nonumber\\
%&&+\frac{n}{2}\sum_{L=3}^N\sum_{j=3}^L\sum_{\sum_l r_l=j\, ;\, \sum_l lr_l=L}\frac{(j-2)\prod_{i=1}^{j-1}(\crit-i)}{\prod_l r_l!}\int_M u_0^{\crit-j}\prod_l P_l(\varphi)^{r_l}\, dv_g\nonumber\\
%&&+o\left(\Vert \varphi\Vert^N\right).\label{eq:4}
%\end{eqnarray}
\begin{align*}
&\sum_{q=1}^{L-1}\frac{L-2q}{L}\sum_{\sum_l s_l=j-1\, ;\, \sum_l ls_l=L-q}\frac{1}{\prod_l s_l!}\int_M u_0^{\crit-j}\Big(\prod_l P_l(\varphi)^{s_l}\Big)P_q(\varphi)\, dv_g\\
&\quad=\sum_{q=1}^{L-1}\frac{L-2q}{L}\sum_{\sum_l r_l=j\, ;\, \sum_l lr_l=L}\frac{r_q}{\prod_l r_l!}\int_M u_0^{\crit-j}\prod_l P_l(\varphi)^{r_l}\, dv_g\\
&\quad=\sum_{\sum_l r_l=j\, ;\, \sum_l lr_l=L}\left(\sum_{q=1}^{L-1}\frac{L-2q}{L}r_q\right)\frac{1}{\prod_l r_l!}\int_M u_0^{\crit-j}\prod_l P_l(\varphi)^{r_l}\, dv_g\\
&\quad=(j-2)\sum_{\sum_l r_l=j\, ;\, \sum_l lr_l=L}\frac{1}{\prod_l r_l!}\int_M u_0^{\crit-j}\prod_l P_l(\varphi)^{r_l}\, dv_g.
\end{align*}
Plugging this identity into \eqref{eq:star} yields
\begin{equation}\label{eq:43}\Vert u_0+\varphi+\phi(\varphi)\Vert_{\crit}^{\crit} = \Vert u_0\Vert_{\crit}^{\crit}+\frac{n}{2}\sum_{L=3}^N\sum_{q=1}^{L-2}\frac{L-2q}{L}u_{L-q,q}(\varphi)+o\left(\Vert \varphi\Vert^N\right),
\end{equation}
where
\begin{multline*}
u_{k, q}(\varphi):=\sum_{j=2}^{k}\left(\prod_{i=1}^{j}(\crit-i)\right)\\
\times\sum_{\sum_l s_l=j\, ;\, \sum_l ls_l=k}\frac{1}{\prod_l s_l!}\int_M u_0^{\crit-1-j}\Big(\prod_l P_l(\varphi)^{s_l}\Big)P_q(\varphi)\, dv_g\,.
\end{multline*}
For any $L,q$ such that $q\geq 2$ and $L-q\geq 2$, the self-adjointness of $L_0$ yields $(L_0 P_q(\varphi), P_{L-q}(\varphi))_h=(P_q(\varphi), L_0P_{L-q}(\varphi))_h$. Taking the explicit expression of \eqref{carac:P:L} then yields
$$u_{L-q, q}(\varphi)=u_{q, L-q}(\varphi)\hbox{ for }2\leq q\leq L-2\,.$$
Therefore, for $L\geq 4$, we get that 
$$\sum_{q=2}^{L-2}\frac{L-2q}{L}u_{L-q,q}(\varphi)=0\,,$$
and then \eqref{eq:43} yields
\begin{eqnarray}
&&\Vert u_0+\varphi+\phi(\varphi)\Vert_{\crit}^{\crit} = \Vert u_0\Vert_{\crit}^{\crit}+\frac{n}{2}\sum_{L=3}^N\frac{L-2}{L}u_{L-1,1}(\varphi)+o\left(\Vert \varphi\Vert^N\right).\label{eq:5}
\end{eqnarray}
We now estimate $\Vert u_0+\varphi+\phi(\varphi)\Vert_{h}^{2}-\Vert u_0+\varphi+\phi\Vert_{\crit}^{\crit}$. Using \eqref{def:psi} and that $u_0,\phi(\varphi)\in K_0^\perp$ for all $\varphi\in K_0$, we get that (writing $\phi=\phi(\varphi)$ for simplicity)
\begin{align}
&\Vert u_0+\varphi+\phi\Vert_{h}^{2}-\Vert u_0+\varphi+\phi\Vert_{\crit}^{\crit} \label{eq:6}\\
&\quad= \big( u_0+\varphi+\phi, u_0+\varphi+\phi-(\Delta_g+h)^{-1}(u_0+\varphi+\phi)^{\crit-1}\big)_h\nonumber\\
&\quad= \big( \Pi_{K_0}(u_0+\varphi+\phi), \Pi_{K_0}(u_0+\varphi+\phi-(\Delta_g+h)^{-1}(u_0+\varphi+\phi)^{\crit-1})\big)_h\nonumber\\
&\quad= \big( \varphi, \varphi-(\Delta_g+h)^{-1}(u_0+\varphi+\phi)^{\crit-1})\big)_h\nonumber\\
&\quad= \Vert \varphi\Vert_h^2-\int_M \left( u_0+\varphi+\phi\right)^{\crit-1}\varphi\, dv_g\,.\nonumber
\end{align}
We fix $N\geq 3$ and write $\phi(\varphi)=\sum_{L=2}^{N-1}P_l(\varphi)+o\left(\Vert \varphi\Vert^{N-1}\right)$ when $\varphi\to 0$. A Taylor expansion and \eqref{poly:expo:k} yield
\begin{multline}
\int_M \left( u_0+\varphi+\phi(\varphi)\right)^{\crit-1}\varphi\, dv_g\label{eq:7}\\
= \int_M u_0^{\crit-1}\varphi\, dv_g+(\crit-1)\int_M u_0^{\crit-2}\varphi^2\, dv_g+\sum_{l=2}^{N-1}(\crit-1)\int_M u_0^{\crit-2}\varphi P_l(\varphi)\, dv_g\\
+\sum_{L=3}^N\sum_{j=2}^{L-1}\left(\prod_{i=1}^{j}(\crit-i)\right)\sum_{\sum_l s_l=j\, ;\, \sum_l ls_l=L-1}\frac{1}{\prod_l s_l!}\int_M u_0^{\crit-1-j}\Big(\prod_l P_l(\varphi)^{s_l}\Big)\varphi\, dv_g\\
+o\left(\Vert \varphi\Vert^{N}\right)
\end{multline}
when $\varphi\to 0$. The definition \eqref{def:K0} of $K_0$ yields 
\begin{equation}\label{eq:phi:norm}
(\crit-1)\int_M u_0^{\crit-2}\varphi^2\, dv_g=\Vert\varphi\Vert_h^2\,.
\end{equation}
Moreover, since $P_l(\varphi)\in K_0^\perp$ for all $l\geq 2$, we get that
\begin{equation}\label{eq:8}
\sum_{l=2}^{N-1}(\crit-1)\int_M u_0^{\crit-2}\varphi P_l(\varphi)\, dv_g=\left(\sum_{l=2}^{N-1}P_l(\varphi),\varphi\right)_h=0\,.
\end{equation}
Plugging together \eqref{ortho:u0:K0} and \eqref{eq:6}--\eqref{eq:8} yields
\begin{equation}
\Vert u_0+\varphi+\phi(\varphi)\Vert_{h}^{2}-\Vert u_0+\varphi+\phi(\varphi)\Vert_{\crit}^{\crit}=-\sum_{L=3}^Nu_{L-1,1}(\varphi)+o\left(\Vert \varphi\Vert^{N}\right)\label{eq:9}
\end{equation}
when $\varphi\to 0$. We define 
\begin{multline}\label{def:A:L}
A_L(\varphi):=-u_{L-1,L}\\
=-\sum_{j=2}^{L-1}\left(\prod_{i=1}^{j}(\crit-i)\right)\sum_{\sum_l s_l=j\, ;\, \sum_l ls_l=L-1}\frac{1}{\prod_l s_l!}\int_M u_0^{\crit-1-j}\Big(\prod_l P_l(\varphi)^{s_l}\Big)\varphi\, dv_g
\end{multline}
which is a homogenous polynomial of degree $L$. Claim \ref{claim:1} then follows from \eqref{eq:5}, \eqref{eq:6}, \eqref{eq:9} and the analyticity of $\varphi\mapsto\phi(\varphi)$ (see Proposition \ref{prop:analycity}).
\endproof

\noindent We define
$$\mathcal{S}_{K_0}:=\{\varphi\in K_0/\, \Vert \varphi\Vert_h=1\}.$$
For any $\varphi\in \mathcal{S}_{K_0}$ and any $t\in (-\nu_0, \nu_0)$, we define
$$f_\varphi(t):=\frac{J_0(u_0+t\varphi+\phi(t\varphi))-J_0(u_0)}{t^2\cdot \Vert u_0\Vert_h^2}\hbox{ if }t\neq 0\hbox{ and }f_\varphi(0)=0\,.$$
It follows from Claim \ref{claim:1} that $f_\varphi$ is analytic on $(-\nu_0,\nu_0)$ and that
$$\Vert u_0+t\varphi+\phi(t\varphi)\Vert_{\crit}^{\crit} = \Vert u_0\Vert_{\crit}^{\crit}\left(1-\frac{n}{2}t^3f_\varphi'(t)\right)$$
for $|t|<\nu_0$. Therefore, we have that
\begin{multline}
I_0(u_0+t\varphi+\phi(t\varphi))\\
=I_0(u_0)\left(1+2t^2 f_\varphi(t)-\frac{n-2}{2}t^3f_\varphi'(t)\right)\cdot\left(1-\frac{n}{2}t^3f_\varphi'(t)\right)^{-\frac{2}{\crit}}.\label{id:I:t:f}
\end{multline}

\begin{claim}\label{claim:2} 
We assume that $u_0$ is a strict local minimizer of $I_0$. Then there exists $\nu_1\in (0,\nu_0)$ such that for any $\varphi\in \mathcal{S}_{K_0}$ and $t\in (-\nu_1,\nu_1)\setminus \{0\}$, there holds
\begin{align*}
f_\varphi(t)=0\; &\Rightarrow\; f_\varphi'(t)\neq 0\,,\\
f_\varphi'(t)=0\; &\Rightarrow\; f_\varphi(t)>0\,.
\end{align*}
\end{claim}

\proof[Proof of Claim \ref{claim:2}] 
If $f_\varphi(t)=f_\varphi'(t)=0$, it then follows from \eqref{id:I:t:f} that $u_0+t\varphi+\phi(t\varphi)$ is a minimizer for $I_0$ close to $u_0$, and therefore there exists $\lambda_t>0$ such that $u_0+t\varphi+\phi(t\varphi)=\lambda_t\cdot u_0$ for $t$ small. It then follows from the definition \eqref{def:psi} of $\phi(t\varphi)$ that $\lambda_t=1$ and that $t\varphi=0$, which is a contradiction since $t\neq 0$ and $\varphi\neq 0$. Therefore $f_\varphi(t)$ and $f_\varphi'(t)$ cannot vanish simultaneously for $t\neq 0$. Moreover, if $f_\varphi'(t)=0$, \eqref{id:I:t:f} yields $f_\varphi(t)\geq 0$. Combining these assertions yields Claim \ref{claim:2}.
\endproof

\begin{claim}\label{claim:3} 
We assume that $u_0$ is a strict local minimizer of $I_0$. We claim that for all $\varphi\in \mathcal{S}_{K_0}$, there exists $\tilde{t}_\varphi\in (0, \nu_1)$ such that $f_\varphi(t)>0$ for all $t\in (0, \tilde{t}_\varphi)$.
\end{claim}

\proof[Proof of Claim \ref{claim:3}] 
It follows from Claim \ref{claim:2} that $f_\varphi$ does not vanish identically. Since it is analytic, there exists $a\neq 0$ and $k\geq 1$ (both depending on $\varphi$) such that $f_\varphi(t)=a t^k+o(t^k)$ when $t\to 0$. Obtaining from this the expansion of $f_\varphi'(t)$ and plugging these expressions into \eqref{id:I:t:f} yield
$$I_0(u_0+t\varphi+\phi(t\varphi))=I_0(u_0)(1+2a t^{k+2}+o(t^{k+2}))$$
when $t\to 0$. Since $u_0$ is a local minimizer, we get that $a\geq 0$, and then $a>0$. This yields the existence of $\tilde{t}_\varphi$. This proves Claim \ref{claim:3}.
\endproof

\noindent It follows from Claims \ref{claim:2} and \ref{claim:3} that for any $\varphi\in \mathcal{S}_{K_0}$, there exists $t_\varphi\in (0,\nu_1]$ such that $f_\varphi(t)>0$ for all $t\in (0,t_\varphi)$, and in case $t_\varphi<\nu_1$, we have that $f_\varphi(t)<0$ for all $t\in (t_\varphi,\nu_1)$.

\begin{claim}\label{claim:4} 
We assume that $u_0$ is a strict local minimizer of $I_0$. We claim that there exists $\nu_2>0$ such that $t_\varphi>\nu_2$ for all $\varphi\in \mathcal{S}_{K_0}$.
\end{claim}

\proof[Proof of Claim \ref{claim:4}] 
We prove Claim \ref{claim:4} by contradiction. Indeed, otherwise, there exists a sequence $(\varphi_i)\in \mathcal{S}_{K_0}$ such that $t_{\varphi_i}\to 0$ when $i\to +\infty$ and $f_{\varphi_i}(t_{\varphi_i})=0$ for all $i$. Up to a subsequence, we can assume that $\varphi_i\to \varphi\in \mathcal{S}_{K_0}$ when $i\to +\infty$. We fix $t\in (0,\nu_1)$. Then for $i$ large enough, we have $t_{\varphi_i}<t$, and therefore $f_{\varphi_i}(t)<0$. Passing to the limit when $i\to +\infty$ yields $f_\varphi(t)\leq 0$ for all $t\in (0,\nu_1)$. This is a contradiction with Claim \ref{claim:3}. This proves Claim \ref{claim:4}.
\endproof

\proof[Proof of Theorem \ref{th:minim}, first implication:] 
We assume that $u_0$ is a strict local minimizer of $I_0$. It follows from Claim \ref{claim:4} that $J_0(u_0+\varphi+\phi(\varphi))>J_0(u_0)$ for all $\varphi\in B_{\nu_2}(0)\setminus\{0\}$. This proves the first implication of Theorem \ref{th:minim}.

\proof[Proof of Theorem \ref{th:minim}, second implication:] 
We assume that there exists $\nu_1>0$ such that $J_0(u_0+\varphi+\phi(\varphi))>J_0(u_0)$ for all $\varphi\in B_{\nu_1}(0)\setminus\{0\}$. For $\varphi\in B_{\nu_1}(0)$, we define $\delta A(\varphi)$ and $\delta B(\varphi)$ such that
$$\Vert u_0+\varphi+\phi(\varphi)\Vert_{h}^{2}=\Vert u_0\Vert_h^2\cdot(1+\delta A(\varphi))\hbox{ and }\Vert u_0+\varphi+\phi(\varphi)\Vert_{\crit}^{\crit}=\Vert u_0\Vert_h^2\cdot(1+\delta B(\varphi)).$$
Therefore, we have that
\begin{align}
J_0(u_0+\varphi+\phi(\varphi))&=J_0(u_0)+\Vert u_0\Vert_h^2\cdot\left(\frac{1}{2}\delta A(\varphi)-\frac{1}{\crit}\delta B(\varphi)\right),\label{eq:J}\\
I_0(u_0+\varphi+\phi(\varphi))&=I_0(u_0)\cdot (1+\delta A(\varphi))\left(1+\delta B(\varphi)\right)^{-2/\crit}\label{eq:I}
\end{align}
for all $\varphi\in B_{\nu_1}(0)$. It follows from our assumption and \eqref{eq:J} that $\delta A(\varphi)>\frac{2}{\crit}\delta B(\varphi)$ for all $\varphi\in B_{\nu_1}(0)\setminus\{0\}$. It then follows from \eqref{eq:I} that 
\begin{equation}\label{min:I}
I_0(u_0+\varphi+\phi(\varphi))>I_0(u_0)\hbox{ for all }\varphi\in B_{\nu_1}(0)\setminus\{0\}.
\end{equation}
We now let $(u_i)\in H_1^2(M)$ be minimizers for $I_0$ such that $\lim_{i\to +\infty} u_i=u_0$. It follows from regularity theory that $u_i\in C^{2,\theta}(M)$ for all $i$ and that the convergence holds in $C^{2,\theta}(M)$. Without loss of generality, we can assume that $u_i$ is a solution to \eqref{eq:u0} for all $i$. It then follows from the definition of $\phi$ (see Proposition \ref{prop:def:phibar}) that there exists $\varphi_i\in K_0$ such that $u_i=u_0+\varphi_i+\phi(\varphi_i)$ for all $i$. Since $u_i$ is a local minimizer, it then follows from \eqref{min:I} that $\varphi_i=0$ for $i$ large, and thus $u_i=u_0$. Then $u_0$ is a strict local minimizer of $I_0$. This proves the second implication of Theorem~\ref{th:minim}.
\endproof

\section{Examples}\label{sec:ex}

In this section, we provide examples of strict local minimizers for the functional $I_0$, and hence for $J_0$ by Theorem \ref{th:minim}. We let $u_0\in C^2(M)$ be a solution to \eqref{eq:u0}. In particular $I_0^\prime(u_0)=0$. As a preliminary remark,
\begin{equation}\label{basic:min}
\hbox{if }u_0\hbox{ is a local minimizer of }I_0\hbox{ then }I_0^{\prime\prime}(u_0)\geq 0.
\end{equation}
Moreover, since $u_0$ is a solution to \eqref{eq:u0}, the kernel of $I_0^{\prime\prime}(u_0)$ is given as follows: for any $f_0\in H_1^2(M)$,
\begin{equation}\label{ker:II}
\left\{I_0^{\prime\prime}(u_0)(f_0,f)=0\hbox{ for all }f\in H_1^2(M)\right\}
\Leftrightarrow\{f_0\in \rr u_0\oplus K_0\}.
\end{equation}
Therefore, $I_0^{\prime\prime}(u_0)$ cannot be positive definite, and a specific analysis along $K_0$ is necessary. It follows from the expression \eqref{def:A:L} of $A_L(\varphi)$ that 
\begin{align}
A_3(\varphi)&=-\frac{(\crit-1)(\crit-2)}{2}\int_M u_0^{\crit-3}\varphi^3\, dv_g\,,\label{eq:A3}\\
A_4(\varphi)&= -(\crit-1)(\crit-2)\bigg(\int_M u_0^{\crit-3}\varphi^2 P_2(\varphi)\, dv_g\label{eq:A4}\\
&\qquad+\frac{\crit-3}{6}\int_M u_0^{\crit-4}\varphi^4\, dv_g\bigg)\nonumber
\end{align}
for all $\varphi\in K_0$. Moreover, it follows from Claim \ref{claim:1} that
\begin{equation}\label{Taylor:I0}
I_0(u_0+\varphi+\phi(\varphi))=I_0(u_0)\cdot\left(1+\frac{2A_3(\varphi)}{3\Vert u_0\Vert_{\crit}^{\crit}}+\frac{A_4(\varphi)}{2\Vert u_0\Vert_{\crit}^{\crit}}+o(\Vert\varphi\Vert^4)\right)
\end{equation}
when $\varphi\to 0$. Therefore, 
\begin{equation}\label{nec:min}
\hbox{if }u_0\hbox{ is a local minimizer of }I_0\hbox{ then }A_3\equiv 0\hbox{ and }A_4(\varphi)\geq 0\hbox{ for all }\varphi\in K_0.
\end{equation}
In the case of the Yamabe equation, this condition appeared in Kobayashi \cite{kobayashi}. Conversely, we have the following result:

\begin{prop}\label{cond:min}
Assume that $A_3\equiv 0$, $I_0^{\prime\prime}(u_0)\geq 0$ and $A_4(\varphi)>0$ for all $\varphi\in K_0\setminus \{0\}$. Then $u_0$ is a strict local minimizer for $I_0$. Moreover, there exists $\nu_1>0$ such that $u_0$ is the only solution to $\Delta_g u+hu=u^{\crit-1}$ in $B_{\nu_1}(u_0)$.
\end{prop}

\proof[Proof of Proposition \ref{cond:min}] We begin with proving the first part of the proposition. We let $\varphi\in K_0$ and $\phi\in (\rr u_0\oplus K_0)^\perp$ be in $H_1^2(M)$. A Taylor expansion yields
\begin{multline}\label{comp:1}
I_0(u_0+\varphi+\phi)=I_0(u_0+\varphi+\phi(\varphi))+I_0^\prime(u_0+\varphi+\phi(\varphi))(\phi-\phi(\varphi))\\
+\frac{1}{2}I_0^{\prime\prime}(u_0)(\phi-\phi(\varphi),\phi-\phi(\varphi))+o\left(\Vert\phi-\phi(\varphi)\Vert_h^2\right)
\end{multline}
as $\varphi,\phi\to 0$. Since $\phi\in (\rr u_0\oplus K_0)^\perp$, $\phi(\varphi)\in K_0^\perp$ and $\phi(\varphi)=O(\Vert \varphi\Vert_h^2)$, we get that 
\begin{equation}\label{comp:2}
\Vert\phi-\phi(\varphi)\Vert_h^2
%&=& \Vert\Pi_{\rr u_0\oplus K_0}(\phi-\phi(\varphi))\Vert_h^2+\Vert\Pi_{(\rr u_0\oplus K_0)^\perp}(\phi-\phi(\varphi))\Vert_h^2\\
%&=&  \Vert\Pi_{\rr u_0\oplus K_0}(\phi(\varphi))\Vert_h^2+\Vert\Pi_{(\rr u_0\oplus K_0)^\perp}(\phi-\phi(\varphi))\Vert_h^2\\
=\Vert\Pi_{(\rr u_0\oplus K_0)^\perp}(\phi-\phi(\varphi))\Vert_h^2+O(\Vert\varphi\Vert_h^4)
\end{equation}
as $\varphi,\phi\to 0$. 
%We then get that
%\begin{eqnarray}
%I_0(u_0+\varphi+\phi)&=&I_0(u_0+\varphi+\phi(\varphi))+I_0^\prime(u_0+\varphi+\phi(\varphi))(\phi-\phi(\varphi))\\
%&&+\frac{1}{2}I_0^{\prime\prime}(u_0)(\phi-\phi(\varphi),\phi-\phi(\varphi))\\
%&&+o\left(\Vert\Pi_{(\rr u_0\oplus K_0)^\perp}(\phi-\phi(\varphi))\Vert_h^2\right)+o(\Vert\varphi\Vert_h^4)
%\end{eqnarray}
%as $\varphi,\phi\to 0$. 
As one can check, for any $u,v\in H_1^2(M)$, $u\not\equiv 0$, we have that
$$I_0^\prime(u)(v)=\frac{2 I_0(u)}{\Vert u\Vert_h^2}\left(u-\frac{\Vert u\Vert_h^2}{\int_M|u|^{\crit}\, dv_g}(\Delta_g+h)^{-1}(F_0'(u)),v\right)_h.$$
Therefore, since $\phi-\phi(\varphi)\in K_0^\perp$, it follows from the definition of $\phi(\varphi)$ in Proposition \ref{prop:def:phibar}, Claim \ref{claim:1} and $A_3\equiv 0$ that
\begin{equation}\label{comp:3}
I_0^\prime(u_0+\varphi+\phi(\varphi))(\phi-\phi(\varphi))=O\left(\Vert \varphi\Vert_h^4\Vert \phi-\phi(\varphi)\Vert_h\right)=o\left(\Vert \varphi\Vert_h^4\right)
\end{equation}
as $\varphi,\phi\to 0$. Since $I_0^{\prime\prime}(u_0)\geq 0$, it follows from \eqref{ker:II} that there exists $c_1>0$ such that
\begin{equation}\label{comp:4}
I_0^{\prime\prime}(u_0)(u,u)\geq 4c_1 \Vert\Pi_{(\rr u_0\oplus K_0)^\perp}(u)\Vert_h^2
\end{equation}
for all $u\in H_1^2(M)$. The positivity of $A_4$ yields the existence of $c_2>0$ such that 
\begin{equation}\label{comp:5}
A_4(\varphi)\geq c_2\Vert\varphi\Vert^4\hbox{ for all }\varphi\in K_0.
\end{equation}
Plugging \eqref{Taylor:I0}, \eqref{comp:2}, \eqref{comp:3}, \eqref{comp:4} and \eqref{comp:5} into \eqref{comp:1} yields the existence of $c_3>0$ such that 
\begin{equation}
I_0(u_0+\varphi+\phi)\geq I_0(u_0)+ c_3 \Vert\varphi\Vert^4+c_1 \Vert\Pi_{(\rr u_0\oplus K_0)^\perp}(\phi-\phi(\varphi))\Vert_h^2
\end{equation}
as $\varphi,\phi\to 0$, where $\varphi\in K_0$ and $\phi\in (\rr u_0\oplus K_0)^\perp$. This proves that $u_0$ is a strict local minimizer of $I_0$.

%\begin{eqnarray}
%I_0(u_0+\varphi+\phi)&\geq &I_0(u_0)+ \frac{I_0(u_0)}{2\Vert u_0\Vert_{\crit}^{\crit}}A_4(\varphi)\\
%&&+2c_1 \Vert\Pi_{(\rr u_0\oplus K_0)^\perp}(\phi-\phi(\varphi))\Vert_h^2\\
%&&+o\left(\Vert\Pi_{(\rr u_0\oplus K_0)^\perp}(\phi-\phi(\varphi))\Vert_h^2\right)+o(\Vert\varphi\Vert_h^4)
%\end{eqnarray}
%as $\varphi,\phi\to 0$.  We then get that there exists $c_3>0$ such that

\medskip\noindent For the second part, for any solution $u\in B_{\nu_1}(u_0)$, we decompose $u:=u_0+\varphi+\psi$ where $\varphi\in K_0$ and $\psi\in K_0^\perp$. We have that $\Vert \varphi\Vert<\nu_1$ and $\Vert \psi\Vert<\nu_1$. It follows from Proposition~\ref{prop:analycity} that if $\nu_1>0$ is small enough, then $\psi=\phi(\varphi)$ and $u=u(\varphi)$.  The positivity of $A_4$ yields the existence of $c>0$ such that $A_4(\varphi)\geq 2c\Vert\varphi\Vert^4$ for all $\varphi\in K_0$. It then follows from Claim \ref{claim:1} that $\Vert u\Vert_h^2-\Vert u\Vert_{\crit}^{\crit}\geq c\Vert\varphi\Vert^4$. Since $u$ is a solution to the equation, we then get that $\varphi=0$ and then $u=u_0$.
\endproof

\noindent In this section, we exhibit situations in which the hypothesis of Proposition \ref{cond:min} hold, which yields strict local minimizers for $I_0$.

\subsection{The expression of $A_4$ when $u_0$ is constant}

We assume here that $h,u_0>0$ are positive constants. In particular, we have that $h=u_0^{\crit-2}$ and that
$$K_0=\{\varphi\in C^{2}(M)/\, \Delta_g\varphi =\lambda\varphi\},$$
where $\lambda:=(\crit-2)u_0^{\crit-2}>0$. In other words, $u_0$ is degenerate if and only if $\lambda$ is an eigenvalue of $\Delta_g$. As one checks, the operator
\begin{equation*}
\begin{array}{cccc}
\Delta_g-\lambda : & K_0^\perp &\to & (K_0^\perp)^\prime\\
 & \phi &\mapsto & \left(\tau\mapsto \int_M((\nabla\phi,\nabla\tau)_g-\lambda \phi \tau)\, dv_g\right)
\end{array}
\end{equation*}
is a bi-continuous isomorphism and then definition \eqref{carac:P:L} yields 
$$P_2(\varphi)=\frac{(\crit-1)(\crit-2)}{2}(\Delta_g-\lambda)^{-1}(u_0^{\crit-3}\varphi^2)$$
for all $\varphi\in K_0$. As a consequence, the expression \eqref{eq:A4} of $A_4$ can be rewritten
\begin{multline}\label{exp:V:constant}
A_4(\varphi)=(\crit-1)(\crit-2)u_0^{\crit-4}\bigg(-\frac{(\crit-1)\lambda}{2}\int_M \varphi^2(\Delta_g-\lambda)^{-1}(\varphi^2)\, dv_g\\
-\frac{\crit-3}{6}\int_M\varphi^4\, dv_g\bigg)
\end{multline}
for all $\varphi\in K_0$.

\subsection{The case of the Yamabe equation on the canonical sphere} 

In the case of the Yamabe equation on the sphere, the kernel $K_0$ parametrizes exactly the noncompact set of minimizers, which makes $A_4$ vanish. More precisely,   

\begin{prop}\label{prop:V:sphere}[Kobayashi \cite{kobayashi}] 
Assume that $(M,g)=(\mathbb{S}^n,\can)$ and that $h\equiv c_n R_{\can}$. Then any solution $u_0$ to \eqref{eq:u0} is minimizing and $A_4\equiv 0$ for all $u_0$. 
\end{prop}

\proof[Proof of Proposition \ref{prop:V:sphere}]
This result is a consequence of Theorem 2.1 in Kobayashi \cite{kobayashi}. We give here an independent proof for the sake of self-content. The vanishing of $A_4$ is a consequence of the direct computation in the proof of (ii) of Proposition~\ref{prop:product} below. We give here a shorter and less technical proof that stresses on properties of solutions to the scalar curvature equation on the sphere
\begin{equation}\label{yam:sphere}
\Delta_{\can}+c_n R_{\can}u=u^{\crit-1}\hbox{ in }\sn.
\end{equation}
The proof relies on two facts: first, the elements of the kernel $K_0$ satisfy a Bianchi--Egnell condition; second, all solutions to \eqref{yam:sphere} minimize $I_0$ (see Obata \cite{obata}).

\medskip\noindent We fix $\varphi\in K_0$. It follows from properties of the canonical sphere (see below) that there exists $t\in\rr \mapsto u(t)$ a smooth function such that $u(t)\in C^{\infty}(\sn)$ is a solution to \eqref{yam:sphere} for all $t$, $u(0)=u_0$ and $u'(0)=\varphi$. This is Bianchi--Egnell condition. Since $u(t)$ is a positive solution to \eqref{yam:sphere}, it follows from Proposition \ref{prop:def:phibar} that for $t$ small, there exists $\varphi(t)\in K_0$ such that $u(t)=u_0+\varphi(t)+\phi(\varphi(t))$. Moreover, $t\mapsto \varphi(t)$ is smooth, $\varphi(0)=0$ and $\varphi'(0)=\varphi$. It follows from \eqref{nec:min} that $A_3\equiv 0$ since $u_0$ minimizes $I_0$. It then follows from the expansion \eqref{Taylor:I0} of $A_4$ that
$$\frac{A_4(\varphi)}{2\Vert u_0\Vert_{\crit}^2}=\lim_{t\to 0}\frac{I_0(u_0+\varphi(t)+\phi(\varphi(t)))-I_0(u_0)}{t^4I_0(u_0)}=\lim_{t\to 0}\frac{I_0(u(t))-I_0(u_0)}{t^4I_0(u_0)}\,.$$
Moreover, it follows from Obata \cite{obata} that positive solutions to \eqref{yam:sphere} are all minimizing, and then $I_0(u(t))=I_0(u_0)$ for all small $t$. Therefore, we get that $A_4(\varphi)=0$ for all $\varphi\in K_0$.

\medskip\noindent We are now left with proving the existence of $t\mapsto u(t)$. Up to conformal transformation (see Obata \cite{obata}), we assume that $u_0$ is the sole positive constant solution to \eqref{yam:sphere}. In this case, $K_0=\{\varphi\in C^2(\sn)\, /\, \Delta_{\can}\varphi=n\varphi\}$
is the space of first spherical harmonics. We fix $\varphi\in K_0$ and we let $Z:=\vec{grad}(\varphi)$ be the associated vector field. This is a conformal vector field and, denoting $f_t$ the associated flow, we have that $f_t^\star\can=\omega(t)^{4/(n-2)}\can$ for some positive function $t\mapsto \omega(t)\in C^\infty(\sn)$ such that $\omega(0)=1$. It follows from the conformal invariance of the scalar curvature equation that $u(t):=\omega(t) u_0$ is also a solution to \eqref{yam:sphere} for all $t$. Moreover, since $f_t^\star\can=\omega(t)^{4/(n-2)}\can$, we have that $\omega'(0)=-\frac{n-2}{2n}\Delta_{\can}\varphi=\frac{n-2}{2n}\hbox{div}_{\can}(Z)=-\frac{n-2}{2}\varphi$, and then $u'(0)=c\varphi$ for some $c\neq 0$. This proves the result after rescaling.
\endproof

\subsection{Product of manifolds and examples of degenerate strict local minimizers}\label{subsec:product}

Let $(M_1,g_1)$ and $(M_2, g_2)$ be two compact manifolds of respective dimensions $d\geq 1$ and $n-d\geq 1$ with $n\geq 3$. We consider the Riemannian manifold $M:=M_1\times M_2$ endowed with the product metric $g:=g_1\oplus g_2$. For $i=1,2$, we let $\lambda_1(M_i, g_i)>0$ be the first nonzero eigenvalue of $\Delta_{g_i}$ on $M_i$. We define
\begin{equation}\label{def:h:product}
h:=\frac{\lambda_1(M_1, g_1)}{\crit-2}\hbox{ and }u_0:=\left(\frac{\lambda_1(M_1, g_1)}{\crit-2}\right)^{\frac{n-2}{4}},
\end{equation}
so that $u_0$ is the only positive constant solution to $\Delta_g u_0+hu_0=u_0^{\crit-1}\hbox{ in }M$. When $d\geq 3$, we define
$$\tilde{h}:=\frac{\lambda_1(M_1, g_1)}{\crit_d-2}\hbox{ and }\tilde{u}_0:=\left(\frac{\lambda_1(M_1, g_1)}{\crit_d-2}\right)^{\frac{d-2}{4}},\hbox{ where }\crit_d:=\frac{2d}{d-2}$$
so that $\tilde{u}_0$ is the only positive constant solution to
$$\Delta_{g_1} \tilde{u}_0+\tilde{h}\tilde{u}_0=\tilde{u}_0^{\crit_d-1}\hbox{ in }M_1.$$
In particular, $\tilde{u}_0$ is a critical point for the functional
$$\tilde{I}_0(u):=\frac{\int_{M_1}\left(|\nabla u|_{g_1}^2+\tilde{h}u^2\right)\, dv_{g_1}}{\left(\int_{M_1}|u|^{\crit_d}\, dv_{g_1}\right)^{\frac{2}{\crit_d}}}$$
for $u\in H_1^2(M_1)\setminus \{0\}$. We prove the following:

\begin{prop}\label{prop:product} 
Let $(M_1,g_1)$ and $(M_2, g_2)$ be two compact manifolds of respective dimensions $d\geq 1$ and $n-d\geq 1$ with $n\geq 3$. We consider the Riemannian manifold $M:=M_1\times M_2$ of dimension $n\geq 3$ endowed with the product metric $g:=g_1\oplus g_2$. We let $h, u_0>0$ be as in \eqref{def:h:product}. We assume that one of the following cases hold:
\begin{enumerate}
\renewcommand{\labelenumi}{(\roman{enumi})}
\item $d\geq 3$, $\lambda_1(M_1, g_1)<\lambda_1(M_2, g_2)$, and  $\tilde{u}_0$ is a local minimizer of $\tilde{I}_0$,
\item $d\geq 1$ and $(M_1,g_1)=(\sd(r), \can)$ with $r>\sqrt{\frac{d}{\lambda_1(M_2,g_2)}}\,.$
\end{enumerate}
Then $u_0$ is a degenerate solution to \eqref{eq:u0} and $I_0^{\prime\prime}(u_0)\geq 0$. Moreover, we have that $A_3(\varphi)=0$ and $A_4(\varphi)>0$ for all $\varphi\in K_0\setminus \{0\}$. In particular, $u_0$ is a strict local minimizer of $I_0$.
\end{prop}

\noindent In the case of the Yamabe equation on the product of spheres, this proposition is a consequence of Kobayashi \cite{kobayashi}.\par

\proof[Proof of Proposition \ref{prop:product}.]
We let $(M_1,g_1)$, $(M_2,g_2)$ be as in the proposition. Since $\lambda_1(\sd(r), \can)=dr^{-2}$ (see Berger--Gauduchon--Mazet \cite{bgm}), we have that
\begin{equation}\label{rel:lambda:1}
\lambda_1(M_1, g_1)<\lambda_1(M_2, g_2)
\end{equation}
in both Cases (i) and (ii). As one checks, 
$$K_0=\{\varphi\in C^2(M)/\, \Delta_g\varphi=\lambda_1(M_1, g_1)\varphi\}.$$
It follows from spectral theory for products that $K_0$ is spanned by the functions $(x,y)\mapsto u_1(x)u_2(y)$ where for $i=1,2$, $u_i: M_i\to \rr$ is an eigenfunction for the eigenvalue $\mu_i$ for $\Delta_{g_i}$, where $\mu_1+\mu_2=\lambda_1(M_1,g_1)$. It then follows from \eqref{rel:lambda:1} that
$$K_0=\{(x,y)\in M\mapsto \varphi(x)/\, \varphi\in \Lambda_1(M_1,g_1)\},$$
where 
$$\Lambda_1(M_1,g_1):=\{\varphi\in C^2(M_1)/\, \Delta_{g_1}\varphi=\lambda_1(M_1, g_1)\varphi\}$$
is the eigenspace associated to the first eigenvalue $\lambda_1(M_1, g_1)$. We claim that 
\begin{equation}\label{cubic:M1}
\int_{M_1}\varphi^3\, dv_{g_1}=0\hbox{ for all }\varphi\in \Lambda_1(M_1,g_1).
\end{equation}
We prove the claim. In Case (i), since $d\geq 3$ and $\tilde{u}_0$ is a local minimizer, \eqref{cubic:M1} follows from \eqref{eq:A3} and \eqref{nec:min}. In Case (ii), since $(M_1,g_1)=(\sd(r),\can)$, $\Lambda_1(M_1,g_1)$ is the restriction to $\sd(r)$ of linear functions on $\rr^{d+1}$, and then \eqref{cubic:M1} follows from symmetry. This proves the claim.
 
\medskip\noindent We claim that $I_0^{\prime\prime}(u_0)\geq 0$.  We prove the claim. Indeed, since $I_0^\prime(u_0)=0$, straightforward computations yield
$$I_0^{\prime\prime}(u_0)(v,v)=\frac{2I_0(u_0)}{\Vert u_0\Vert_h^2}\left(\Vert v\Vert_h^2-(\crit-1)\int_M u_0^{\crit-2}v^2\, dv_g\right)$$
for all $v\in (\rr u_0)^\perp\subset H_1^2(M)$. With the choice of $u_0$ and $h$ in \eqref{def:h:product}, we then get that
$$I_0^{\prime\prime}(u_0)(v,v)=\frac{2I_0(u_0)}{\Vert u_0\Vert_h^2}\left(\int_M |\nabla v|_g^2\, dv_g-\lambda_1(M_1,g_1)\int_M v^2\, dv_g\right)$$
for all $v\in (\rr u_0)^\perp\subset H_1^2(M)$. Therefore, it follows from the above characterization of $K_0$ above that $I_0^{\prime\prime}(u_0)(v,v)\geq 0$ for all $v\in (\rr u_0)^\perp$. Since $u_0$ is in the kernel of $I_0^{\prime\prime}(u_0)$, we then get that $I_0^{\prime\prime}(u_0)\geq 0$. This proves the claim.

\medskip\noindent Since the elements of $K_0$ are independent of the second variable, we get that
\begin{equation*}
(\Delta_g-\lambda_1(M_1, g_1))^{-1}((x,y)\mapsto \varphi^2(x))=(x,y)\mapsto (\Delta_{g_1}-\lambda_1(M_1, g_1))^{-1}(\varphi^2(x))
\end{equation*}
for all $\varphi\in \Lambda_1(M_1,g_1)$ where $(\Delta_{g_1}-\lambda_1(M_1, g_1))^{-1}$ is the inverse of the isomorphism
\begin{equation*}
\begin{array}{ccc}
 \Lambda_1(M_1, g_1)^\perp &\to & (\Lambda_1(M_1, g_1)^\perp)^\prime\\
 \phi &\mapsto & \left(\tau\mapsto \int_{M_1}((\nabla\phi,\nabla\tau)_{g_1}- \lambda_1(M_1, g_1)\phi \tau)\, dv_{g_1}\right)
\end{array}
\end{equation*}
where the orthogonality in $H_1^2(M_1)$ is considered with respect to the standard $L^2-$product. As a consequence, the expression \eqref{exp:V:constant} can be rewritten
\begin{multline}\label{exp:V:product:2}
A_4(\varphi)=\frac{c_1\hbox{Vol}_{g_2}(M_2)}{2}\bigg(-(\crit-1)\lambda_1(M_1, g_1)\int_{M_1} \varphi^2(\Delta_{g_1}-\lambda_1(M_1, g_1))^{-1}(\varphi^2)\, dv_{g_1}\\
-\frac{\crit-3}{3}\int_{M_1}\varphi^4\, dv_{g_1}\bigg)
\end{multline}
for all $\varphi\in K_0$, where $c_1:=(\crit-1)(\crit-2)u_0^{\crit-4}$ and, for simplicity, we have written $K_0$ for $\Lambda_1(M_1, g_1)$. We now distinguish Cases (i) and (ii) of Proposition \ref{prop:product}:

\medskip\noindent{\bf Case (i): $d\geq 3$ and $\tilde{u}_0$ is a local minimizer.} As one checks, 
$$\tilde{K}_0:=\left\{\varphi\in C^{2}(M_1)/\, \Delta_{g_1} \varphi+\tilde{h}\varphi=(\crit_d-1)\tilde{u}_0^{\crit_d-2}\varphi\right\}=\Lambda_1(M_1,g_1).$$
We define $\tilde{A}_4$ for $\tilde{u}_0$ and therefore \eqref{exp:V:constant} yields
\begin{multline*}
\frac{2}{(\crit_d-1)(\crit_d-2)\tilde{u}_0^{\crit_d-4}}\tilde{A}_4(\varphi)\\
=-(\crit_d-1)\lambda_1(M_1, g_1)\int_{M_1} \varphi^2(\Delta_{g_1}-\lambda_1(M_1, g_1))^{-1}(\varphi^2)\, dv_{g_1}\\-\frac{\crit_d-3}{3}\int_{M_1}\varphi^4\, dv_{g_1}
\end{multline*}
for all $\varphi\in \Lambda_1(M_1,g_1)$. Plugging this expression into \eqref{exp:V:product:2} yields
\begin{equation*}
A_4(\varphi)=c_2\cdot\left(\frac{(\crit-1)\tilde{A}_4(\varphi)}{4(\crit_d-1)^2(\crit_d-2)\tilde{u}_0^{\crit_d-4}} +\frac{(n-d)}{3(n-2)(d+2)}\int_{M_1}\varphi^4\, dv_{g_1}\right)
\end{equation*}
for all $\varphi\in \Lambda_1(M_1,g_1)$, where $c_2:=4(\crit-1)(\crit-2)u_0^{\crit-4}\hbox{Vol}_{g_2}(M_2)$. In particular, if $\tilde{u}_0$ is a local minimizer for $\tilde{I}_0$, then \eqref{nec:min} yields $\tilde{A}_4\geq 0$. Therefore, $A_4(\varphi)>0$ for all $\varphi\in \Lambda_1(M_1,g_1)\setminus\{0\}$ since $n-d>0$. This proves Proposition \ref{prop:product} in Case (i).

\medskip\noindent{\bf Case (ii): $(M_1, g_1)=(\mathbb{S}^d(r),\can)$.} The case $d\geq 3$ is covered by Case (i), and only the cases $d=1,2$ remain to be covered. For simplicity, we assume that $r=1$. It follows from Berger--Gauduchon--Mazet \cite{bgm} that  the second positive eigenvalue $\lambda_2(\mathbb{S}^d,\can)$ is $2(d+1)$ and the eigenfunctions are the restrictions to $\sd$ of second-order homogeneous harmonic polynomials on $\rr^{d+1}$.

\medskip\noindent We let $\Eucl$ be the Euclidean metric on $\rr^{d+1}$. We claim that
\begin{equation}\label{eq:inv:sphere}
(\Delta_{\can}-\lambda_1)^{-1}(\varphi^2)=\frac{\varphi^2+\frac{\lambda_2\Delta_{\Eucl}(\varphi^2)}{2(d+1)\lambda_1}}{\lambda_2-\lambda_1}\hbox{ for all }\varphi\in \Lambda_1(\mathbb{S}^d,\can).
\end{equation}
where $\lambda_1=d$ and $\lambda_2=2(d+1)$. We prove the claim. We fix $\varphi\in \Lambda_1(\mathbb{S}^d,\can)$. In particular $\varphi^2$ is a second-order homogeneous polynomial on $\rr^{d+1}$, and $\varphi^2+\frac{\Delta_{\Eucl}(\varphi^2)}{2(d+1)}|x|^2$ is a harmonic second-order homogenous polynomial, and therefore its restriction to $\sd$ is an eigenfunction for $\lambda_2$. Since $\Delta_{\Eucl}(\varphi^2)$ is constant and $|x|^2$ is constant on $\sd$, \eqref{eq:inv:sphere} follows from a direct computation. This proves the claim.

\medskip\noindent We claim that
\begin{equation}\label{id:phi:4}
\int_{\sd}\varphi^4\, dv_{\can}=-\frac{3}{2(d+3)}\Delta_{\Eucl}(\varphi^2)\int_{\sd}\varphi^2\, dv_{\can}\hbox{ for all }\varphi\in \Lambda_1(\mathbb{S}^d,\can).
\end{equation}
We prove the claim. Since, up to homothetic transformation, $\varphi$ is a coordinate function, proving \eqref{id:phi:4} is equivalent to proving $\int_{\sd}x^4\, dv_{\can}=(3/(d+3))\int_{\sd}x^2\, dv_{\can}$ where $x$ is the first coordinate in $\rr^{d+1}$. This latest identity follows from the change of variable  $(t,\sigma)\mapsto (t, \sqrt{1-t^2}\sigma)$ from $(-1,1)\times\mathbb{S}^{d-1}$ to $\sd\setminus\{(\pm 1,...,0)\}$. This proves the claim.

\medskip\noindent Plugging \eqref{eq:inv:sphere} and \eqref{id:phi:4} into \eqref{exp:V:product:2} yields
\begin{equation*}
A_4(\varphi)=\frac{4(\crit-1)(\crit-2)u_0^{\crit-4}\hbox{Vol}_{g_2}(M_2)(n-d)}{3(n-2)(d+2)}\int_{\sd}\varphi^4\, dv_{\can}
\end{equation*}
for all $\varphi\in \Lambda_1(\mathbb{S}^d,\can)$. In particular, since $d<n$, we have that $A_4(\varphi)>0$ for all $\varphi\in \Lambda_1(\mathbb{S}^d,\can)\setminus\{0\}$. This proves Case (ii) of Proposition \ref{prop:product} when $r=1$. The general case follows by rescaling. This proves Proposition \ref{prop:product}.
\endproof

\noindent As a remark, the computations made for Case (ii) are valid when $d=n\geq 3$ (that is $M=\sd=\sn$), and we get that $A_4\equiv 0$, which has been obtained by another method in Proposition \ref{prop:V:sphere}.

\medskip\noindent When $h\equiv c_n R_g$, an immediate consequence of Proposition \ref{prop:product} is the following:

\begin{coro}\label{th:type:4:yam} 
Let $(N,g_N)$ be a compact Riemannian manifold of positive constant scalar curvature. We choose $d\geq 1$ and we assume that 
\begin{equation}\label{cond:prop:type:4}
R_{g_N}<\hbox{dim}(N)\lambda_1(N,g_N)\hbox{ and }n:=d+\hbox{dim}(N)\geq 3\,.
\end{equation}
We endow the manifold $M:=\sd\big(\sqrt{\hbox{dim}(N)\cdot d/R_{g_N}}\big)\times N$ with the product metric $g:=\can\oplus g_N$.  Then the positive constant solution to the scalar curvature equation $\Delta_g u+c_n R_g u=u^{\crit-1}$ on $M$ is a degenerate strict local minimizer. 
\end{coro}
\proof[Proof of Corollary \ref{th:type:4:yam}.] We fix $r_0:=\sqrt{\hbox{dim}(N)\cdot d/R_{g_N}}$. It follows from inequality \eqref{cond:prop:type:4} that we are in Case (ii) of Proposition \ref{prop:product}. With this choice of $r_0$, we have that
$$c_n R_g=\frac{n-2}{4(n-1)}\left(R_{g_N}+\frac{d(d-1)}{r_0^2}\right)=\frac{(n-2)d}{4r_0^2}=\frac{\lambda_1(\sd(r_0),\can)}{\crit-2}.$$
Therefore Proposition \ref{prop:product} applies. This proves Corollary \ref{th:type:4:yam}.\endproof

\noindent Inequality \eqref{cond:prop:type:4} holds if $g_N$ is a Yamabe metric, that is a minimizer of the Yamabe functional. From the pde point of view, a metric $g$ on $M$ is a Yamabe metric iff $R_g$ is constant and the minimum of $I_0$ (with $h\equiv c_n R_g$) is achieved by constants. 

\medskip\noindent As a remark, Corollary \ref{th:type:4:yam} can be generalized by replacing the sphere by a manifold $V$ of dimension $d\geq 3$ with a degenerate Yamabe metric $g_V$ of positive scalar curvature satisfying $R_{g_N}=\hbox{dim}(N)\lambda_1(V,g_V)$ and $\lambda_1(V,g_V)<\lambda_1(N,g_N)$. Note that the degeneracy of $g_V$ implies that $R_{g_V}=(\hbox{dim}(V)-1)\lambda_1(V,g_V)$.

\addtolength{\textheight}{20pt}

\end{document}